\documentclass[reqno,12pt]{amsart}
\usepackage{bookmark}
\begin{document}

\bibliographystyle{amsplain}

\newtheorem{theorem}{Theorem}[section]
\newtheorem{prop}[theorem]{Proposition}
\newtheorem{lemma}[theorem]{Lemma}
\newtheorem{definition}[theorem]{Definition}
\newtheorem{corollary}[theorem]{Corollary}
\newtheorem{example}[theorem]{Example}
\newtheorem{remark}[theorem]{Remark}
\newcommand{\ra}{\rightarrow}
\numberwithin{equation}{section}

\newcommand{\norm}[1]{\lVert#1\rVert}
\newcommand{\Ag}[1]{\langle#1\rangle}
\setlength{\parskip}{.5em}

\def\xint#1{\mathchoice
  {\xxint\displaystyle\textstyle{#1}}%
  {\xxint\textstyle\scriptstyle{#1}}%
  {\xxint\scriptstyle\scriptscriptstyle{#1}}%
  {\xxint\scriptscriptstyle\scriptscriptstyle{#1}}%
  \!\int}
\def\xxint#1#2#3{{\setbox0=\hbox{$#1{#2#3}{\int}$}
  \vcenter{\hbox{$#2#3$}}\kern-.5\wd0}}
\def\ddashint{\xint=}
\def\dashint{\xint-}

\def \RD {{\mathbb R}^d}
\def \RM {{\mathbb R}^m}
\def \HS {H^1_{w}({\mathbb R}^n)}
\def\HSL  {H^1_{L,w}({\mathbb R}^n)}
\def \HL {H^1_{L,w}({\mathbb R}^n)}
\def \HHL {\mathbb{H}^1_{L,w}({\mathbb R}^n)}
\def \LP {L^p_w}
\def\ls{\lesssim}
\def\e{\varepsilon}
\def\brd{\mathbb{R}^d}

\title[Oscillatory integrals and homogenization]{Oscillatory integrals and periodic homogenization of Robin boundary value problems}

\thanks{{\it Key words and phrase:} Oscillatory integrals; Robin boundary condition; Homogenization.}

\author[J. Geng, and J. Zhuge]
{Jun Geng \qquad Jinping Zhuge}


\begin{abstract}
In this paper, we consider a family of second-order elliptic systems subject to a periodically oscillating Robin boundary condition. We establish the qualitative homogenization theorem on any Lipschitz domains satisfying a non-resonance condition. We also use the quantitative estimates of oscillatory integrals to obtain the dimension-dependent convergence rates in $L^2$, assuming that the domain is smooth and strictly convex.
 \end{abstract}

 \medskip

\maketitle

\tableofcontents

\section {Introduction}
\setcounter{equation}{0}

The primary purpose of this paper is to study a family of elliptic systems with oscillating Robin boundary condition in a bounded Lipschitz domain $\Omega\subset \RD, d\ge 3$, 
\begin{equation}\label{Robin}
\left\{
\aligned
\mathcal{L}_\e (u_\e) &= F \quad &\text{ in }& \Omega,\\
\frac{\partial u_\varepsilon}{\partial \nu_\varepsilon}+b(x/\varepsilon)u_\e  &=g \quad &\text{ on }& \partial\Omega,
\endaligned
\right.
\end{equation}
where 
\begin{align}\label{e1.1}
\mathcal{L}_\varepsilon =-{\rm div}\bigg[A\Big(\frac{x}{\varepsilon}\Big)\nabla\bigg] =-\frac{\partial}{\partial x_i}
\bigg[a_{ij}^{\alpha\beta}\Big(\frac{x}{\varepsilon}\Big)
\frac{\partial}{\partial x_j}\bigg], \qquad \varepsilon>0,
\end{align}
and $u_\e = (u_\e^1,u_\e^2, \cdots, u_\e^m)$ and $1\le \alpha, \beta \le m, 1\le i,j \le d$. The summation convention is used throughout this paper. We assume that the coefficients $A = (a^{\alpha\beta}_{ij})$ and $b = (b^{\alpha\beta})$ satisfy the following conditions
\begin{itemize}
	\item Ellipticity condition: there exists $\mu>0$ such that
	\begin{equation}\label{cond.ellipticity}
	\mu|\xi|^2\leq a_{ij}^{\alpha\beta}(y)\xi_i^{\alpha}\xi_j^\beta\leq \mu^{-1} |\xi|^2,
	\qquad  {\rm for}  \ y\in \RD, \xi=(\xi_i^\alpha)\in \mathbb{R}^{m\times d},
	\end{equation}
	and
	\begin{equation}\label{cond.elliptic.b}
	\mu|\eta|^2\leq b^{\alpha\beta}(y)\eta^{\alpha}\eta^\beta\leq \mu^{-1} |\eta|^2,
	\qquad  {\rm for}  \ y\in \RD, \eta=(\eta^\alpha)\in \RM.
	\end{equation}
	\item Periodicity condition: for any $y\in \RD, z\in {\mathbb Z}^d$,
	\begin{equation}\label{cond.periodicity}
	A(y+z)=A(y),
	\end{equation}
	and
	\begin{equation}\label{cond.periodicity.b}
	b(y+z) = b(y).
	\end{equation}
	\item Continuity: 
	\begin{equation}\label{cond.continuity}
	\text{The coefficient matrix } b \text{ is continuous in } \RD.
	\end{equation}
\end{itemize}

The equation (\ref{Robin}) with $m = 1$ may be used to model, for example, the heat conductivity problems of heterogeneous materials surrounded by a certain fluid. On the surface of the material, a Robin boundary condition is needed due to the convection or phase transition. Actually, a general convective boundary condition is described as
\begin{equation}\label{cond.convective}
\text{heat flux} = b(T - T_{\text{liquid}}), \qquad \text{on } \partial\Omega,
\end{equation}
where $T$ is the temperature of the material surface, $T_{\text{liquid}}$ is the temperature of the fluid and $b$ is known as the heat transfer coefficient established by measurements (depending both on the material and liquid). Under the conditions (\ref{cond.ellipticity}) -- (\ref{cond.continuity}) and assuming $T_{\text{liquid}}$ is a constant and $u_\e = T -  T_{\text{liquid}}$, we know the heat flux is given by $-\frac{\partial u_\e}{\partial \nu_\e}$ and $b$ takes a form of $b(x/\e)$, which is periodic. Therefore, the boundary condition (\ref{cond.convective}) may be rewritten as
\begin{equation*}
\frac{\partial u_\e}{\partial \nu_\e} + b(x/\e) u_\e = 0.
\end{equation*}
If $T_{\text{liquid}}$ is not a constant and could be measured near the material surface, then $(\ref{cond.convective})$ may be reduced to a nontrivial oscillating Robin boundary data that takes a form of $g(x,x/\e)$ and $g(x,y)$ is 1-periodic with respect to $y$. This case may also be handled by the method in this paper.

Recall that the variational form of equation (\ref{Robin}) may be written as
\begin{equation}\label{eq.variation}
\int_{\Omega} A(x/\e) \nabla u_\e \cdot \nabla \phi + \int_{\partial\Omega} b(x/\e) u_\e \cdot \phi = \int_{\partial\Omega} g \cdot \phi + \int_{\Omega} F \cdot \phi,
\end{equation}
for any $\phi \in H^1(\Omega;\RM)$. Then, the Lax-Milgram Theorem implies the existence and uniqueness of the weak solution of (\ref{Robin}), provided $F\in H^{-1}(\Omega;\RM)$ and $g\in H^{-1/2}(\partial\Omega;\RM)$. Moreover, the sequence $\{ u_\e: \e>0 \}$ is uniformly bounded in $H^1(\Omega;\RM)$. Now, we may ask a standard question in homogenization theory: as $\e \to 0$, do the solutions $u_\e$ converge to some function $u_0$ in $H^1(\Omega;\RM)$, while $u_0$ is the weak solution of the effective equation with a certain effective Robin boundary condition? It turns out that the answer to this question is quite different from the usual Dirichlet or Neumann boundary value problems, due to the oscillating factor of $b(x/\e)$ on the boundary. We mention that the Dirichlet and Neumann problems have been well-studied in many literatures and we refer to a recent excellent monograph \cite{ShenBook} by Z. Shen.

The purpose of this paper is two-fold: (i) prove the qualitative homogenization theorem under very weak assumptions on the coefficients and domain; (ii) obtain the $L^2$ convergence rates under sufficiently strong assumptions.

\subsection{Qualitative homogenization}
Our method of proving the qualitative homogenization theorem is based on the direct asymptotic analysis of the variational equation (\ref{eq.variation}). In view of the second integral on the left-hand side of (\ref{eq.variation}), the key for homogenization to take place is the asymptotic behavior of the following oscillatory integral
\begin{equation}\label{eq.OscInt}
\int_S f(x, \lambda x) d\sigma(x), \qquad \text{as }\lambda \to \infty, 
\end{equation}
where $S = \partial\Omega$, $d\sigma$ is the surface measure on $S$ (namely, $(d-1)$-dimensional Hausdoff measure) and $f(x,y)$ is a 1-periodic function in $y\in \RD$. The asymptotic behavior of the integral (\ref{eq.OscInt}) may be derived through the classical oscillatory integral theory if $S$ is smooth and possesses certain geometric conditions. For example, if $S$ is smooth and strictly convex and $f$ is smooth, then (\ref{eq.OscInt}) converges with error $O(\lambda^{\frac{1-d}{2}})$ as $\lambda \to \infty$. Recently, S. Kim and K.-A. Lee showed in \cite{KLS17} that if $S$ is a $C^1$ surface and satisfies a non-resonance condition (which is called irrational direction dense condition in \cite{KLS17}), then the integral (\ref{eq.OscInt}) converges qualitatively as $\lambda \to 0$. In this paper, we further weaken the regularity of $S$ to any Lipschitz surfaces, which is sharp in the sense that $d\sigma$ is well-defined and the following non-resonance condition makes sense.

\begin{definition}\label{def.non-resonance}
	Let $S$ be a Lipschitz surface. We say $S$ satisfies the non-resonance condition with respect to $\mathbb{Z}^d$ (or simply non-resonance condition), if
	\begin{equation*}
	\sigma( \{x\in S: n(x) \text{ is well-defined and } n(x) \in \mathbb{RZ}^d \} ) = 0.
	\end{equation*}
	Here $n(x)$ is the unit normal to $S$ at $x$. (A vector $a\in \RD$ is called rational if $a\in \mathbb{R} \mathbb{Z}^d$.)
\end{definition}

The non-resonance condition in Definition \ref{def.non-resonance}, generalizing the notion in \cite{KLS17}, covers all the classes of domains people innovated to deal with PDEs subject to periodically oscillating boundary value, including strictly convex smooth domains \cite{GerMas12,AKMP16,ShenZhuge16,ASS15}, finite-type domains \cite{Zhuge18}, polygonal domains \cite{GerMas11,ASS14} and more \cite{Fe14,CK14} (also see Remark \ref{rmk.layered} for a possible extension for layered or directional materials). Under this non-resonance condition, we are able to prove Theorem \ref{thm.mainOscInt}, a (qualitative) version of the Weyl's equidistribution theorem, based on the fundamental ideas from \cite{GerMas12} and \cite{KLS17}. Among many potential applications of the Weyl's equidistribution theorem in analysis and PDEs (see, e.g., Theorem \ref{thm.Laplace.L2}), we prove the qualitative homogenization theorem for (\ref{Robin}).
\begin{theorem}\label{thm.main}
	Let $(A,b)$ satisfy (\ref{cond.ellipticity}) - (\ref{cond.continuity}). Assume $\Omega$ is a bounded Lipschitz domain satisfying the non-resonance condition with respect to $\mathbb{Z}^d$. Suppose that $u_\e$ is the weak solution of (\ref{Robin})
	with $F\in H^{-1}(\Omega;\RM)$ and $g\in H^{-1/2}(\partial\Omega;\RM)$. Then, there exists some function $u_0\in H^1(\Omega;\RM)$ such that, as $\e \to 0$,
	\begin{equation}\label{eq.ueTOu0}
	\left\{
	\aligned
	u_\e &\rightarrow u_0  &\quad &\text{weakly in } H^1(\Omega;\RM) \text{ and strongly in } L^2(\Omega;\RM), \\
	A_\e(x/\e)  \nabla u_\e & \rightarrow \widehat{A} \nabla u_0 & \quad & \text{weakly in } L^2(\Omega;\mathbb{R}^{m\times d}),
	\endaligned
	\right.
	\end{equation} 
	where $\widehat{A}$ is the homogenized coefficient matrix. Moreover, $u_0$ is the weak solution of
	\begin{equation}\label{eq.A0u}
	\left\{
	\aligned
	-{\rm div}(\widehat{A} \nabla u_0) &= F &\quad& \text{ in } \Omega,\\
	\frac{\partial u_0}{\partial \nu_0}+\bar{b} u_0 & =g &\quad &\text{ on } \partial\Omega,
	\endaligned
	\right.
	\end{equation}
	where $\frac{\partial}{\partial \nu_0} = n\cdot \widehat{A} \nabla $ and $\bar{b} = \int_{\mathbb{T}^d} b(y)dy$.
\end{theorem}

Theorem \ref{thm.main} is a straightforward corollary of Theorem \ref{thm.Quali.Homo} which is proved in a more general setting. Note that the effective heat transfer coefficient in (\ref{eq.A0u}) is the average of $b$. This fact may be extended without any real difficulty to more general Robin boundary conditions, such as
\begin{equation*}
\frac{\partial u_\e}{\partial \nu_\e} + b(x,x/\e) u_\e = g(x,x/\e),
\end{equation*}
where $b(x,y)$ and $g(x,y)$ are $1$-periodic with respect to $y$. In this case, the homogenized boundary condition reads
\begin{equation*}
\frac{\partial u_0}{\partial \nu_0} + \bar{b}(x) u_0 = \bar{g}(x),
\end{equation*}
where $\bar{b}(x) = \int_{\mathbb{T}^d} b(x,y)dy$ and $\bar{g}(x) = \int_{\mathbb{T}^d} g(x,y)dy$.


\subsection{Convergence rates}
The second part of the paper is concerned with the convergence rates of $u_\e$ to $u_0$ in $L^2(\Omega;\RM)$, which seems to be of more interest. The sharp convergence rates for the Dirichlet and Neumann problems of elliptic systems 
have been well-studied by a standard framework; see, e.g., \cite{CDG08,Sus1,Sus2,SZ17}. In particular, it has been established for both the Dirichlet and Neumann problems that
\begin{equation*}
\norm{u_\e - u_0}_{L^2(\Omega)} \le C\e \norm{u_0}_{H^2(\Omega)},
\end{equation*}
for any dimension $d\ge 2$ and any bounded $C^{1,1}$ domains. However, the convergence rates for the Robin problem (\ref{Robin}) is quite different due to the oscillating factor $b(x/\e)$ on the boundary which essentially requires an additional strong geometric assumption on $\partial\Omega$ in order to carry out a quantitative analysis. In this paper, we will try to obtain the best possible convergence rates under the strongest geometric assumption that is commonly used in the quantitative analysis of the oscillating boundary layers \cite{GerMas12,AKMP16,ShenZhuge16,ASS15}, namely, $\Omega$ is smooth and strictly convex (in the sense that all the principle curvatures of $\partial\Omega$ are positive).

The following is our main result.
\begin{theorem}\label{thm.L2Rate}
	Let $(A,b)$ satisfy (\ref{cond.ellipticity}) -- (\ref{cond.periodicity.b}). In addition, assume $b \in C^\infty(\mathbb{T}^d;\mathbb{R}^{m\times m})$ and $\Omega$ is a smooth and strictly convex domain. Let $u_\e$ and $u_0$ be the same as before and $u_0$ be sufficiently smooth. Then, for any $\sigma \in (0,1)$, there exists $C>0$ such that
	\begin{equation}\label{est.L2Rates}
	\norm{u_\e - u_0}_{L^2(\Omega)} \le \left\{
	\aligned
	&C\e^{\frac{7}{8} - \sigma } \big( \norm{u_0}_{W^{1,\infty}(\Omega)} + \norm{u_0}_{H^2(\Omega)} \big) \qquad &\text{if } d =3, \\
	&C\e^{1-\sigma}\big( \norm{u_0}_{W^{1,\infty}(\Omega)} + \norm{u_0}_{H^2(\Omega)} \big)  \qquad &\text{if } d = 4,\\
	&C\e^{\frac{3}{4} + \frac{3}{4(d-1)} - \sigma } \norm{u_0}_{W^{2,d/2}(\Omega)}  \qquad &\text{if } d \ge 5.
	\endaligned
	\right.
	\end{equation}
\end{theorem}

It is crucial to point out that, in the above theorem, the convergence rate for $d=4$ is nearly sharp since $\sigma>0$ can be arbitrarily small; but the rates are possibly not nearly sharp for other dimensions, especially for $d\ge 5$. The intrinsic reasons for the appearance of the dimension-dependent exponents in (\ref{est.L2Rates}) will be explained later. Nevertheless, we provide an effective framework in this paper that allows further improvements. On the other hand, if we consider the estimates of $\norm{u_\e - u_0}_{L^p(\Omega)}$ with some $p\in (1,2)$ depending on $d$, then the almost sharp convergence rate $O(\e^{1-\sigma})$ may still be obtained for all dimensions; see Remark \ref{rmk.Lprate}. Finally, we mention that, to be simple and more concentrated, we will not try to optimize the norms of $u_0$ on the right-hand side of (\ref{est.L2Rates}).

The proof of Theorem \ref{thm.L2Rate} follows a line as the usual Dirichlet or Neumann problems (see, e.g., \cite{Shen17,ShenBook}), but relies essentially on the quantitative estimates of certain oscillatory integrals. Let $w_\e = u_\e - u_0 - \e \chi(x/\e) S_\e(  \eta_\e\nabla u_0)$ be the error of the first-order approximation of $u_\e$ (which is the same as Dirichlet or Neumann problems). Then a direct computation shows that
\begin{equation}\label{eq.we.var}
\begin{aligned}
\int_{\Omega} A_\e \nabla w_\e \cdot \nabla \varphi + \int_{\partial\Omega} b_\e w_\e \cdot \varphi &= \int_{\partial\Omega} (\overline{b}^{\alpha\beta} - b^{\alpha\beta}(x/\e)) u_0^\beta \cdot \varphi^\alpha d\sigma \\
& \qquad + \text{other familiar integrals over } \Omega, 
\end{aligned}
\end{equation}
where $\varphi \in H^1(\Omega;\RM)$ is a test function. The key to (\ref{eq.we.var}) is the quantitative estimate of the oscillatory integral on $\partial\Omega$ in the following form
\begin{equation}\label{eq.osc.f}
\int_{\partial \Omega} f(x/\e)\cdot \phi(x) d\sigma,
\end{equation}
where $f(y)$ is a $1$-periodic vector-value function with zero mean. At first sight, one may expect to employ the oscillatory integral theory to analyze the quantitative behavior of (\ref{eq.osc.f}) as $\e \to 0$, which is feasible if $\phi$ is smooth enough. However, in our application, $\phi = u_0 \varphi$ is only supposed to be in certain Sobolev spaces that a direct using of oscillatory integral theory will only give rough estimates. To establish the convergence rates as better as possible, in this paper, we will develop a ``duality approach'' to analyze (\ref{eq.osc.f}) via an auxiliary Neumann problem. Precisely, let $v_\e$ be the solution of the following Neumann problem
\begin{equation}\label{eq.ve.N}
\left\{
\aligned
-\text{div}(\widehat{A} \nabla v_\e) & =0 \quad &\text{ in }& \Omega,\\
n\cdot \widehat{A} \nabla v_\e & =f(x/\e) - M_\e \quad &\text{ on }& \partial\Omega.
\endaligned
\right.
\end{equation}
where $\widehat{A}$ is the (constant) homogenized matrix and $M_\e$ is a constant vector such that the compatibility condition is satisfied. Note that (\ref{eq.ve.N}) may be formally viewed as a simplified version of the Robin problem (\ref{Robin}). By the integration by parts, it is not hard to see that the estimate of (\ref{eq.osc.f}) is reduced to the estimates of $v_\e$; see (\ref{eq.osc.dual}). Unlike using the oscillatory integral theory to (\ref{eq.osc.f}) directly, our method completely get rid of the difficulty caused by the lower regularity of $\phi$.

The Neumann problem (\ref{eq.ve.N}), of independent interest itself, has been studied in \cite{ASS15} by H. Aleksanyan, H. Shahgholian and P. Sj\"{o}lin. In particular, they obtained the estimates of $\norm{v_\e}_{L^1(\Omega)}$ and $\norm{\nabla v_\e}_{L^1(\Omega)}$ for system (\ref{eq.ve.N}); see Theorem \ref{thm.ve.g}. To obtain better estimates of $\norm{v_\e}_{L^p(\Omega)}$ with general $p>1$, which are critical in our application, we establish a new estimate of $\norm{v_\e}_{L^\infty(\Omega)}$ in this paper, namely,
\begin{equation}\label{est.ve.infity}
\norm{v_\e}_{L^\infty(\Omega)} \le C\e^{\frac{1}{2}}.
\end{equation}
This estimate is proved by using the integral representation for (\ref{eq.ve.N}) and taking advantage of the lower singularity of the Neumann function (compared to Poisson kernel). 
Then, combing these estimates, we conclude by interpolation that (see Theorem \ref{thm.ve.improve})
\begin{equation}\label{est.veL2}
	\norm{v_\e }_{L^2(\Omega)} \le \left\{
	\aligned
	&C\e^{\frac{7}{8}- \sigma} \qquad &\text{if } d =3, \\
	&C\e \qquad &\text{if } d= 4,\\
	&C\e^{\frac{5}{4}} |\ln \e|^{\frac{1}{2}} \qquad &\text{if } d\ge 5.
	\endaligned
	\right.
\end{equation}
Note that the estimates above are worse for lower dimensions due to the nature of oscillatory integrals with non-degenerate phases. In particular, we point out that the rate $\e^{\frac{7}{8}-\sigma}$ for $d=3$ follows naturally by an interpolation of $\norm{\nabla v_\e}_{L^1(\Omega)} \le C\e^{1-\sigma}$ and (\ref{est.ve.infity}). This explains why the convergence rate for $d=3$ in Theorem \ref{thm.L2Rate} cannot exceed $\e^{\frac{7}{8}-\sigma}$.

On the other hand, the estimate of the integral on the right-hand side of (\ref{eq.we.var}) definitely depends on the regularity of $u_0$ and $\varphi$. Under the assumption $u_0\in H^2(\Omega;\RM)$, we have
\begin{equation}\label{est.bu0.varphi}
\bigg|\int_{\partial\Omega} (\overline{b} - b(x/\e) -M_\e) u_0 \cdot \varphi d\sigma \bigg| \le \left\{
\aligned
&C\e^{\frac{2}{3}- \sigma}\norm{u_0}_{H^2(\Omega)} \norm{\varphi}_{H^1(\Omega)} \quad &\text{if } d =3, \\
&C\e^{\frac{1}{2}-\sigma}\norm{u_0}_{H^2(\Omega)} \norm{\varphi}_{H^1(\Omega)} \quad &\text{if } d = 4,\\
&C\e^{\frac{1}{4} + \frac{3}{4(d-1)} - \sigma }\norm{u_0}_{H^2(\Omega)} \norm{\varphi}_{H^1(\Omega)} \quad &\text{if } d \ge 5.
\endaligned
\right.
\end{equation}
See Lemma \ref{lem.Iosc.H1} (ii) for details. The above estimates are worse (than $\e^{\frac{1}{2}}$) for $d\ge 5$, because $u_0\in H^2(\Omega;\RM)$ actually implies worse integrability of $u_0$ and $\nabla u_0$ if $d$ is larger; see Remark \ref{rmk.u0H2}. Consequently, by a duality argument, the $L^2$ convergence rates may be raised to $\e^{\frac{3}{4} + \frac{3}{4(d-1)} - \sigma}$ for $d\ge 5$ as shown in Theorem \ref{thm.L2Rate}.

Finally, we would like to emphasize that even though the estimates of $\norm{v_\e}_{L^1(\Omega)}$ and $\norm{\nabla v_\e}_{L^1(\Omega)}$ in Theorem \ref{thm.ve.g}, as well as (\ref{est.ve.infity}), are possibly optimal in view of the singularity of the Neumann function, it is still open whether (\ref{est.veL2}) and (\ref{est.bu0.varphi}) are optimal, which are very basic questions related to singular oscillatory integrals (see (\ref{eq.Sing.Osc})) and of independent interest. Clearly, through our framework, any further improvement of (\ref{est.veL2}) for $d = 3$ or (\ref{est.bu0.varphi}) for $d\ge 5$ will lead to an corresponding improvement for Theorem \ref{thm.L2Rate}.

\subsection{Organization of the paper}
We prove a Weyl's equidistribution theorem and Theorem \ref{thm.main} in Section 2 and 3, respectively. In Section 4, we obtain the quantitative estimates of the auxiliary Neumann system (\ref{eq.ve.N}). In Section 5, we prove Theorem \ref{thm.L2Rate}.

\section{An Equidistribution Theorem}
This section is devoted to the proof of a (qualitative) version of the Weyl's equidistribution theorem under the non-resonance condition.
\begin{theorem}\label{thm.mainOscInt}
	Let $S$ be a closed Lipschitz surface satisfing the non-resonance condition with respect to $\mathbb{Z}^d$. Assume $f(x,y) : S\times \mathbb{T}^d \mapsto \mathbb{R}$ is 1-periodic and continuous in $y$ for each $x\in S$. Moreover, assume
	\begin{equation*}
		\int_{S} \norm{f(x,\cdot)}_{C(\mathbb{T}^d)} d\sigma(x) < \infty.
	\end{equation*}
	Then
	\begin{equation}\label{est.Ergodic.fxy}
		\lim_{\lambda \to \infty} \int_{S} f(x,\lambda x) d\sigma(x) = \int_{S} \int_{\mathbb{T}^d}  f(x,y) dy d\sigma(x).
	\end{equation}
\end{theorem}

For a given unit vector $\xi \in \mathbb{S}^{d-1}$, denote by $H_\xi : = \{x\in \RM: \xi\cdot x = 0 \}$ the hyperplane perpendicular to $\xi$.

\begin{lemma}\label{lem.Ergodic.Plane}
	Let $\xi \in \mathbb{S}^{d-1}$ be irrational. If $g \in C(\mathbb{T}^d)$ is a 1-periodic continuous function and $f \in L^1(H_\xi)$. Then
	\begin{equation}\label{eq.Irr.Ergodic}
	\lim_{\lambda \to \infty} \sup_{y\in \RD} \bigg| \int_{H_\xi} g(y+ \lambda x) f(x) d\sigma(x) - \Ag{g} \int_{H_\xi} f(x) d\sigma(x) \bigg| = 0.
	\end{equation}
\end{lemma}
\begin{proof}
	First of all, it follows from the classical ergodic property of quasi-periodic functions that
	\begin{equation}\label{eq.ergo1}
	\lim_{\lambda \to \infty} \int_{H_\xi} g(\lambda x) f(x) d\sigma(x) = \Ag{g} \int_{H_\xi} f(x) d\sigma(x),
	\end{equation}
	for 1-periodic $g\in C^\infty(\mathbb{T}^d)$ and $f \in C_0^\infty(H_\xi)$, where $\Ag{g} = \int_{\mathbb{T}^d} g(y)dy$. Then the general case with $g \in C(\mathbb{T}^d)$ and $f \in L^1(H_\xi)$ follows by an approximation argument. Now, we prove (\ref{eq.Irr.Ergodic}) by contradiction. If (\ref{eq.Irr.Ergodic}) is not true, since $g$ is 1-periodic, there exist $\delta>0$, and sequences of $y_k\in \mathbb{T}^d$ and $\lambda_k \in \mathbb{R}$ such that $\lim_{k\to \infty}\lambda_k = \infty$ and
	\begin{equation}\label{eq.ergo.contra}
	\bigg| \int_{H_\xi} g(y_k+ \lambda_k x) f(x) d\sigma(x) - \Ag{g} \int_{H_\xi} f(x) d\sigma(x) \bigg| > \delta, \quad \text{for any } k\in \mathbb{N}.
	\end{equation}
	Now, by the compactness of $\{ g(y+\cdot): y\in \mathbb{T}^d \}$ in $C(\mathbb{T}^d)$, we may choose a subsequence of $\{k_\ell \} \subset \mathbb{N}$ such that $y_{k_\ell} \to z \in \mathbb{T}^d$ and $g(y_{k_\ell}+\cdot) \to g(z+\cdot)$ in $C(\mathbb{T}^d)$, as $k_\ell \to \infty$. Observe that
	\begin{equation}\label{est.ergo2}
	\begin{aligned}
	&\bigg| \int_{H_\xi} g(y_{k_\ell}+ \lambda_{k_\ell} x) f(x) d\sigma(x) - \Ag{g} \int_{H_\xi} f(x) d\sigma(x) \bigg| \\
	& \quad \le \bigg| \int_{H_\xi} [ g(y_{k_\ell}+ \lambda_{k_\ell} x) - g(z+\lambda_{k_\ell} x) ] f(x) d\sigma(x) \bigg| \\
	& \qquad + \bigg| \int_{H_\xi}  g(z+\lambda_{k_\ell} x) f(x) d\sigma(x) - \Ag{g} \int_{H_\xi} f(x) d\sigma(x) \bigg|.
	\end{aligned}
	\end{equation}
	Clearly, the first integral in the right-hand side of (\ref{est.ergo2}) is bounded by $\norm{g(y_{k_\ell} + \cdot) - g(z + \cdot)}_{C(\mathbb{T}^d)} \norm{f}_{L^1(H_\xi)} $, which tends to zero as $k_\ell \to \infty$. Also, (\ref{eq.ergo1}) implies that the second integral in the right-hand side of (\ref{est.ergo2}) tends to zero as $k_\ell \to \infty$. It follows that
	\begin{equation*}
	\bigg| \int_{H_\xi} g(y_{k_\ell}+ \lambda_{k_\ell} x) f(x) d\sigma(x) - \Ag{g} \int_{H_\xi} f(x) d\sigma(x) \bigg| \to 0, \quad \text{as } k_\ell \to \infty.
	\end{equation*}
	This contracts with (\ref{eq.ergo.contra}) and hence proves (\ref{eq.Irr.Ergodic}).
\end{proof}
	
Now, we may use Lemma \ref{lem.Ergodic.Plane} to prove a simplified version of Theorem \ref{thm.mainOscInt}.
\begin{theorem}\label{thm.Ergodic.S}
	Let $S$ be a closed Lipschitz surface and satisfy the non-resonance condition with respect to $\mathbb{Z}^d$. Then if $g\in C(\mathbb{T}^d)$ and $f\in L^1(S)$,
	\begin{equation}\label{eq.Ergodic.main}
	\lim_{\lambda \to \infty} \int_{S} g(\lambda x) f(x) d\sigma(x) = \Ag{g} \int_{S} f(x) d\sigma(x).
	\end{equation}
\end{theorem}

\begin{proof}
	{\bf Step 1: Reduction.}
	First of all, we may assume $g$ and $f$ are smooth. The general case follows by an approximation argument. By a partition of unity, we may restrict ourself to a local coordinate system (parallel to the original coordinate system so that the non-resonance condition is still valid with respect to $\mathbb{Z}^d$) such that for some $x_0 = (x_0',x_{0,d})\in S$,  $B(x_0,R_0) \cap S$ is the graph of
	\begin{equation*}
	x_d = \phi(x') \quad \text{for } x'\in Q(x_0',R_0) \subset \mathbb{R}^{d-1},
	\end{equation*}
	where $Q(x'_0,R_0)$ is a ($d-1$)-dimensional cube parallel to coordinates and centered at $x_0'$ with radius $R_0$. Put $Q_0 = Q(x'_0,R_0)$. Without loss of generality, we assume $|Q_0| = 1$. Thus, it suffices to concentrate on
	\begin{equation*}
	\begin{aligned}
	&\int_{\{(x',\phi(x')): x' \in Q_0  \}} g(\lambda x) f(x) d\sigma(x) \\
	& = \int_{Q_0} g(\lambda(x',\phi(x'))) f(x',\phi(x')) \sqrt{1+|\nabla \phi(x')|^2 } dx'
	\end{aligned}
	\end{equation*}
	We may reset $f(x',\phi(x')) \sqrt{1+|\nabla \phi(x')|^2 }$ to be $f_0(x')$, and assume again that $f_0$ is smooth. Then, we only need to show
	\begin{equation*}
	\lim_{\lambda \to \infty} \int_{Q_0} g(\lambda(x',\phi(x'))) f_0(x') dx' = \Ag{g} \int_{Q_0} f_0(x') dx'.
	\end{equation*}
	Or, more specifically, we would like to show that for any given $\eta>0$, there exists $\lambda_\eta >0$ such that for any $\lambda> \lambda_\eta$,
	\begin{equation}\label{est.ergodic.eta}
	\bigg| \int_{Q_0} g(\lambda(x',\phi(x'))) f_0(x') dx' - \Ag{g} \int_{Q_0} f_0(x') dx'\bigg| \le \eta.
	\end{equation}
	
	{\bf Step 2: Ruling out bad points.}
	By our assumption, $\phi(x')$ is a Lipschitz function and thus $\nabla \phi(x') $ is well-defined for almost every $x'\in Q_0$. Let $E \subset Q_0$ be the set that $\nabla \phi$ is not defined. Now, for $x'\in Q_0\setminus E$, define
	\begin{equation*}
	J(x',r) = \sup_{|y'-x'|\le r} \frac{|\phi(y') - \phi(x') - \nabla \phi(x') \cdot (y'-x')|}{|y'-x'|}
	\end{equation*}
	Since $\phi(x')$ is differentiable at $x'$, then
	\begin{equation*}
	\lim_{r\to 0} J(x',r) = 0, \quad \text{for any } x'\in Q_0\setminus E.
	\end{equation*}
	Now, by the Egoroff's theorem, for a given $\delta>0$ (to be determined), there exists a measurable subset $F \subset Q_0, E\subset F$ such that $|F| < \delta/2$ and
	\begin{equation*}
	J(x',r) \text{ converges to 0 uniformly for any } x'\in Q_0\setminus F.
	\end{equation*}	
	Define
	\begin{equation*}
	\omega_1(r) = \sup_{x'\in Q_0\setminus F} J(x',r).
	\end{equation*}
	Then, $\omega_1(r)$ is a increasing function and
	\begin{equation*}
	\lim_{r\to 0} \omega_1(r) = 0.
	\end{equation*}

	Next, we are going to use the non-resonance condition. Let $n(x')$ be the unit vector normal vector of the graph $x_d = \phi(x')$ at $(x',x_d)$, if there exists. Let
	\begin{equation*}
	G = \{(x',x_d)\in S: n(x') \text{ is well-defined and } n(x') \in \mathbb{R}\mathbb{Z}^d \}.
	\end{equation*}
	Clearly, by the non-resonance condition with respect to $\mathbb{Z}^d$, $|G| = 0$. Moreover, for $x'\in Q_0 \setminus G$, we know
	\begin{equation*}
	n(x') = \frac{(\nabla \phi(x'), -1)}{\sqrt{1+|\nabla \phi(x')|^2}} \ \text{ is irrational.}
	\end{equation*}
	Thus, we may define a good subset of $Q_0$,
	\begin{equation*}
	U = Q_0\setminus (F \cup G).
	\end{equation*}
	
	{\bf Step 3: Decomposition.}
	We construct a family of dyadic cubes in $Q_0$. Let $\{Q_k^j: j =1,2,3,\cdots, 2^{dk}\}$ be a collection of dyadic cubes at level $k$. Put
	\begin{equation*}
	P_k = \{ Q_k^j: U \cap Q_k^j \neq \emptyset \}.
	\end{equation*}
	
	Before we proceed, we claim that, there exists a decreasing step function $\rho(\lambda)$, taking discrete values in $\{ 2^{-k}: k\in \mathbb{N} \}$, such that $\rho(\lambda) \to 0$ as $\lambda \to \infty$. Moreover,	
	\begin{equation}\label{est.rho.condition}
	\lim_{\lambda \to \infty} \lambda \rho(\lambda) \omega_1( \sqrt{d} \rho(\lambda) ) = 0 \quad \text{and} \quad \lim_{\lambda \to \infty}\lambda \rho(\lambda) = \infty.
	\end{equation}
	
	The existence of such $\rho(\lambda)$ follows from a concrete construction. Actually, let
	\begin{equation}\label{eq.rho.def}
	\rho(\lambda):= \sup\Big\{ 2^{-k}: k\in \mathbb{N}, \lambda 2^{-k} \sqrt{\omega_1(\sqrt{d} 2^{-k})} \le 1 \Big\}.
	\end{equation}
	It is not hard to see $\rho(\lambda)$ is a well-defined decreasing step function, since $r \sqrt{\omega_1(\sqrt{d} r)}$, as a function of $r$, is increasing. Moreover, $\rho(\lambda) \to 0$ as $\lambda \to \infty$. Now, note that (\ref{eq.rho.def}) implies
	\begin{equation*}
	\lambda \rho(\lambda) \sqrt{\omega_1(\sqrt{d} \rho(\lambda))} \le 1.
	\end{equation*}
	It follows
	\begin{equation*}
	\lambda \rho(\lambda) \omega_1( \sqrt{d} \rho(\lambda) ) \le \sqrt{\omega_1(\sqrt{d} \rho(\lambda))} \to 0, \quad \text{as } \lambda \to \infty.
	\end{equation*}
	
	To see the second part of (\ref{est.rho.condition}), by the definition (\ref{eq.rho.def}), we know
	\begin{equation}\label{est.rho.ge1}
	\lambda 2 \rho(\lambda) \sqrt{\omega_1(\sqrt{d} 2 \rho(\lambda))} > 1.
	\end{equation}
	which yields
	\begin{equation*}
	\lambda \rho(\lambda) > \frac{1}{2 \sqrt{ \omega_1(\sqrt{d}2 \rho(\lambda))}} \to \infty, \quad \text{as } \lambda \to \infty.
	\end{equation*}
	This proves the claim as desired.
	
	Now, given any $\lambda\ge 1$, let $\rho(\lambda)$ be the decreasing step function as above such that $\rho(\lambda) = 2^{-k(\lambda)}$, where $k(\lambda) \in \mathbb{N}$. For $x'\in U$,
	\begin{equation*}
	Q(x',\lambda) = \text{the dyadic cube in } P_{k(\lambda)} \text{ containing } x'.
	\end{equation*}
	Note that $\rho(\lambda)$ is the side length of $Q(x',\lambda)$. Then, for each $\lambda$, $U$ may be covered by a sequence of dyadic cubes $Q(x',\lambda)$ with $x'\in U$, at level $k(\lambda)$.
	
	{\bf Step 4: Local approximation.}
	With $\rho(\lambda)$ and $Q(x',\lambda)$ as constructed above, we claim that for any $x'\in U$,
	\begin{equation}\label{est.Qrho}
	\lim_{\lambda \to 0}  \bigg| \dashint_{Q(x',\lambda)} g(\lambda(y',\phi(y'))) f_0(y') dy' - \Ag{g} \dashint_{Q(x',\lambda)} f_0(y') dy' \bigg| = 0 ,
	\end{equation}
	
	To prove (\ref{est.Qrho}), we temporarily fix $x'$ and $\lambda$ and consider
	\begin{equation}\label{eq.Kx'lambda}
	K(x',\lambda) := \dashint_{Q(x',\lambda)} g(\lambda(y',\phi(y'))) f_0(y') dy'.
	\end{equation}
	By the Taylor's expansion at $x'$, for $y'\in Q(x',\lambda)$
	\begin{equation}\label{eq.Taylor}
	\phi(y') = \phi(x') + \nabla \phi(x')(y'-x') + R(y',x')
	\end{equation}
	where, by the fact $x'\in Q_0\setminus F$ and the definition of $\omega_1$, the remainder satisfies
	\begin{equation*}
	|R(y',x')| \le |y'-x'|\omega_1(|y'-x'|).
	\end{equation*}
	Thus, by the smoothness of $g$, we have
	\begin{equation}\label{est.gExp}
	\begin{aligned}
	\big|g(\lambda(y',\phi(y'))) - g(\lambda (y', \phi(x') + \nabla \phi(x')\cdot (y'-x') ))  \big| &\le  L_g|y'-x'| \omega_1 (|y'-x'|) \\
	& \le L_g \sqrt{d} \rho(\lambda) \omega_1( \sqrt{d} \rho(\lambda) ),
	\end{aligned}
	\end{equation}
	where $L_g = \sup |\nabla g|$. Substituting (\ref{eq.Taylor}) and (\ref{est.gExp}) into (\ref{eq.Kx'lambda}), we have
	\begin{equation}\label{est.K}
	\begin{aligned}
	&\bigg|K(x',\lambda) - \dashint_{Q(x',\lambda)} g(\lambda(y',\phi(x') + \nabla \phi(x')(y'-x'))) f_0(x') dy' \bigg| \\
	& \le L_g M_f \sqrt{d} \lambda \rho(\lambda) \omega_1( \sqrt{d} \rho(\lambda) ) + M_g L_f \sqrt{d} \rho(\lambda),
	\end{aligned}
	\end{equation}
	where $M_f = \sup |f_0|, M_g = \sup |g|$ and $L_f = \sup|\nabla f_0|$. Observe that the above error is independent of $x'$. Thus, (\ref{est.K}) tends to zero as $\lambda \to \infty$ by the construction of $\rho(\lambda)$ and the first part of (\ref{est.rho.condition}).
	
	Next, we consider
	\begin{equation*}
	\begin{aligned}
	I_1(x',\lambda) & := f_0(x') \dashint_{Q(x',\lambda)} g(\lambda(y',\phi(x') + \nabla \phi(x')\cdot (y'-x')))  dy' \\
	& = f_0(x') \dashint_{Q(x',\lambda)} g(\lambda((x',\phi(x')) + ( y'-x', \nabla \phi(x')\cdot (y'-x')))  dy'.
	\end{aligned}
	\end{equation*}
	By a change of variable $w = (w',w_d)$ such that
	\begin{equation*}
	w' = \frac{y'-x'}{\rho(\lambda)} \quad \text{and} \quad w_d = \nabla \phi(x')\cdot w',
	\end{equation*}
	we see
	\begin{equation*}
	I_1(x',\lambda) = f_0(x') \dashint_{\widetilde{Q}(x',\lambda)} g(\lambda (x',\phi(x')) + \lambda \rho(\lambda) w)  d\sigma(w)
	\end{equation*}
	where $\widetilde{Q}(x',\lambda)$, the image of $Q(x',\lambda)$ under $w$, is some parallel polyhedron on the irrational hyperplane (perpendicular to the irrational direction $(\nabla\phi (x'), -1)$, since $x'\in U$). The position of $\widetilde{Q}(x',\lambda)$ may depends on $\lambda$; but its shape and size, depending only on $\nabla \phi(x')$, are fixed. Moreover, $|\widetilde{Q}(x')| = (\sqrt{1+|\nabla \phi(x')|^2})^{-1}$. Since we have $\lambda \rho(\lambda) \to \infty$ as $\lambda \to \infty$, by Lemma \ref{lem.Ergodic.Plane},
	\begin{equation}\label{est.I1.limit}
	\lim_{\lambda \to \infty } |I_1(x',\lambda) - \Ag{g} f_0(x')| = 0.
	\end{equation}
	
	On the other hand, it is clear that
	\begin{equation}\label{est.fz'}
	\lim_{\lambda\to 0}\bigg|\Ag{g} f_0(x') - \Ag{g} \dashint_{Q(x',\lambda)} f_0(y') dy' \bigg| = 0,
	\end{equation}
	by the smoothness of $f_0$.
	
	Finally, combining (\ref{est.K}), (\ref{est.I1.limit}) and (\ref{est.fz'}), we obtain (\ref{est.Qrho})
	
	{\bf Step 5: Completing the proof.} Define
	\begin{equation*}
	T(x',\lambda) := \bigg| \dashint_{Q(x',\lambda)} g(\lambda(y',\phi(y'))) f_0(y') dy' - \Ag{g} \dashint_{Q(x',\lambda)} f_0(y') dy' \bigg|
	\end{equation*}
	It has been prove in Step 4 that
	\begin{equation*}
	\lim_{\lambda\to \infty} T(x',\lambda) = 0, \quad \text{for any } x'\in U.
	\end{equation*}
	By the Egoroff's theorem, for a given $\delta>0$, there exists $H \subset U$ such that $|H| \le \delta/2$ and
	\begin{equation}\label{est.Lx'lambda}
	\omega_2(\lambda): = \sup_{x'\in V} T(x',\lambda) \to 0 \quad \text{as } \lambda \to 0,
	\end{equation}
	where $V = U\setminus H$. Note that $|Q_0\setminus V| \le \delta$.
	
	Now given $\eta>0$, choose $\delta>0$ small enough so that
	\begin{equation}\label{est.delta.eta}
	\delta < \frac{\eta}{4 M_f M_g}.
	\end{equation}
	By (\ref{est.Lx'lambda}), we can choose $\lambda_\eta $ large enough such that
	\begin{equation}\label{est.L.eta}
	T(x',\lambda)< \eta/2 \qquad \text{for any } x'\in V \text{ and } \lambda>\lambda_\eta.
	\end{equation}
	For any given $\lambda>\lambda_\eta$, there exists a collection of $Q(x',\lambda), x'\in V$ such that
	\begin{equation*}
	V \subset \bigcup_{x'\in V} Q(x',\lambda).
	\end{equation*}
	Therefore, we can find finite $Q(x'_j,\lambda)$ with $1\le j \le N_0$ and $x'_j\in V$, such that
	\begin{equation}\label{eq.W.Qj}
	V \subset W:=\bigcup_{1\le j\le N_0} Q(x'_j,\lambda) \quad \text{and} \quad Q(x'_i,\lambda)\cap Q(x'_j,\lambda) = \emptyset, \text{ if } i\neq j.
	\end{equation}
	Obviously,
	\begin{equation}\label{est.Wdelta}
	|Q_0 \setminus W| \le |Q_0 \setminus V| \le \delta.
	\end{equation}
	
	It  follows that
	\begin{equation*}
	\begin{aligned}
	&\bigg| \int_{Q_0} g(\lambda(x',\phi(x'))) f_0(x') dx' - \Ag{g} \int_{Q_0} f_0(x') dx'\bigg| \\
	& \le \bigg| \int_{Q_0\setminus W} g(\lambda(x',\phi(x'))) f_0(x') dx' - \Ag{g} \int_{Q_0 \setminus W} f_0(x') dx'\bigg| \\
	& \qquad + \sum_{j=1}^{N_0} \bigg| \int_{Q(x'_j,\lambda)} g(\lambda(x',\phi(x'))) f_0(x') dx' - \Ag{g} \int_{Q(x'_j,\lambda)} f_0(x') dx'\bigg| \\
	& \le 2 M_f M_g |Q_0\setminus W| +  \sum_{j=1}^{N_0}|Q(x'_j,\lambda)| T(x'_j,\lambda)\\
	& \le 2\delta M_f M_g + \frac{\eta |W|}{2} \\
	& \le \eta,
	\end{aligned}
	\end{equation*}
	where we have used (\ref{est.L.eta}), (\ref{eq.W.Qj}) and (\ref{est.Wdelta}) in the third inequality and used (\ref{est.delta.eta}) in the last inequality. This finishes the proof of (\ref{est.ergodic.eta}).
\end{proof}

\begin{proof}[Proof of Theorem \ref{thm.mainOscInt}]
	We use an approximation argument. Let $\phi \in C_0^\infty(B_1(0))$ such that $\int \phi =1$ and let $\phi_\delta(x) = \delta^{-d} \phi(\delta^{-1} x)$ with $0<\delta<1$. Define
	\begin{equation*}
	f_\delta(x,y) = \int_{\RD} f(x,y-z) \phi_{\delta}(z) dz.
	\end{equation*}
	Then, it is easy to see that
	\begin{equation}\label{est.fd.Cm}
	\norm{f_\delta(x,\cdot)}_{C^m(\mathbb{T}^d)} \le C_{m,\delta} \norm{f(x,\cdot)}_{C(\mathbb{T}^d)}, \qquad \text{for any } x\in S, m\in \mathbb{N},
	\end{equation}
	and
	\begin{equation}\label{est.fd-f}
	|f_\delta(x,y) - f(x,y)| \le C\omega(x,\delta), \qquad \text{for any } (x,y) \in S\times \mathbb{T}^d,
	\end{equation}
	where $\omega(x,\delta)$ is the modulo of the continuity of $f(x,\cdot)$, and $C$ is independent of $\delta$. By the fact $\omega(x,\delta) \to 0$ as $\delta \to 0$ for each $x\in S$, and the dominant convergence theorem, one obtains
	\begin{equation*}
	\lim_{\delta \to 0} \int_{S} \omega(x,\delta) d\sigma(x) = 0.
	\end{equation*}
	Consequently, in view of (\ref{est.fd-f}), to see (\ref{est.Ergodic.fxy}), it suffices to show
	\begin{equation}\label{est.fxy.delta}
	\lim_{\lambda \to \infty} \int_{S} f_\delta(x,\lambda x) d\sigma(x) = \int_{S} \int_{\mathbb{T}^d}  f_\delta(x,y) dy d\sigma(x).
	\end{equation}
	
	By Fourier series expansion of $f_\delta(x,y)$ in terms of $y$, we have
	\begin{equation*}
	f_\delta(x,y) = \sum_{k\in \mathbb{Z}^d} a_k(x) e^{2\pi i k \cdot y}, \qquad \text{where } a_k(x) = \int_{\mathbb{T}^d} f_\delta(x,y) e^{2\pi i k \cdot y} dy.
	\end{equation*}
	In view of (\ref{est.fd.Cm}), we have
	\begin{equation*}
	|a_k(x)| \le C_{m,\delta} |k|^{-m} \norm{f(x,\cdot)}_{C(\mathbb{T}^d)}, \qquad \text{for every } m\in \mathbb{N}.
	\end{equation*}
	It follows that
	\begin{equation*}
	\int_{S} f_\delta(x,\lambda x) d\sigma(x) = \sum_{|k| \le M} \int_{S} a_k(x) e^{2\pi i \lambda k\cdot x} d\sigma(x) + \sum_{|k| > M} \int_{S} a_k(x) e^{2\pi i \lambda k\cdot x} d\sigma(x).
	\end{equation*}
	Observe that the second term in the last equality is bounded by
	\begin{equation*}
	C_{d+1,\delta} \sum_{|k|> M} |k|^{-(d+1)} \int_{S} \norm{f(x,\cdot)}_{C(\mathbb{T}^d)} d\sigma(x) \le \frac{C_{d+1,\delta}}{M} \int_{S} \norm{f(x,\cdot)}_{C(\mathbb{T}^d)} d\sigma(x).
	\end{equation*}
	On the other hand, by Theorem \ref{thm.Ergodic.S}, we know
	\begin{equation*}
	\lim_{\lambda \to \infty} \sum_{|k| \le M} \int_{S} a_k(x) e^{2\pi i \lambda k\cdot x} d\sigma(x) = \int_{S} a_0(x) d\sigma(x) = \int_{S} \int_{\mathbb{T}^d}  f_\delta(x,y) dy d\sigma(x).
	\end{equation*}
	As a result,
	\begin{equation*}
	\bigg| \lim_{\lambda \to \infty} \int_{S} f_\delta(x,\lambda x) d\sigma(x) - \int_{S} \int_{\mathbb{T}^d}  f_\delta(x,y) dy d\sigma(x) \bigg| \le \frac{C_{d+1,\delta}}{M} \int_{S} \norm{f(x,\cdot)}_{C(\mathbb{T}^d)} d\sigma(x).
	\end{equation*}
	Since $M$ may be chosen arbitrarily large, the above estimate implies (\ref{est.fxy.delta}) by letting $M\to \infty$. This ends the proof.
\end{proof}

Observe that the convergence in (\ref{eq.Ergodic.main}) may depend on the function $f(x,y)$. However, one may have a fixed convergence rate, provided some compactness on the underlying function spaces. Let 
\begin{equation*}
L^1(S; C(\mathbb{T}^d) ) := \{ f(x,y):  \int_{S} \norm{f(x,\cdot)}_{C(\mathbb{T}^d)} d\sigma(x) < \infty\}.
\end{equation*}

\begin{theorem}\label{thm.rateXY}
	Let $X$ be a Banach space compactly embedded into $L^1(S; C(\mathbb{T}^d) )$, endowed with norm $\norm{\cdot}_X$. Then there exists a rate function $\omega(\lambda)$, depending only on $d,S$ and $X$, such that $\omega(\lambda) \downarrow 0$ as $\lambda \to \infty$, and
	\begin{equation}\label{est.fgXY}
	\bigg| \int_{S} f(x,\lambda x) d\sigma(x) - \int_{S} \int_{\mathbb{T}^d}  f(x,y) dy d\sigma(x) \bigg| \le \omega(\lambda) \norm{f}_{X},
	\end{equation}
	for any $f\in X$.
\end{theorem}
\begin{proof}
	Let $B_X = \{f\in X: \norm{f}_X \le 1  \}$. To prove (\ref{est.fgXY}), it suffices to show
	\begin{equation}\label{eq.w.lambda}
	\omega(\lambda):= \sup_{f\in B_X} \bigg| \int_{S} f(x,\lambda x) d\sigma(x) - \int_{S} \int_{\mathbb{T}^d}  f(x,y) dy d\sigma(x) \bigg| \to 0, \quad \text{as } \lambda \to \infty.
	\end{equation}	
	
	The above statement may be proved by contradiction. Suppose (\ref{eq.w.lambda}) is not true. Then there exist $\delta>0$ and a sequence $\{ \lambda_k: k\in \mathbb{N}\}$ such that $\lim_{k\to \infty} \lambda_k = \infty$ and $\omega(\lambda_k) > \delta$. This implies that there exists a sequence $\{f_k: k\in \mathbb{N}\} \subset B_X$ such that, for any $k>0$
	\begin{equation}\label{est.S.Contra}
	\bigg|  \int_{S} f_k(x,\lambda_k x) d\sigma(x) - \int_{S} \int_{\mathbb{T}^d}  f_k(x,y) dy d\sigma(x) \bigg| > \frac{\delta}{2}.
	\end{equation}
	By the compactness of $B_X \subset L^1(S; C(\mathbb{T}^d) )$, we may choose a subsequence $\{ k_\ell \} \subset \mathbb{N}$ and $f_\infty \in L^1(S; C(\mathbb{T}^d) )$ such that $f_{k_\ell} \to f_\infty$ in $L^1(S; C(\mathbb{T}^d) )$, as $k_\ell \to \infty$. Now, observe that
	\begin{equation*}
		\begin{aligned}
			&\bigg| \int_{S} f_{k_\ell}(x,\lambda_{k_\ell} x) d\sigma(x) - \int_{S} \int_{\mathbb{T}^d}  f_{k_\ell}(x,y) dy d\sigma(x) \bigg| \\
			& \le \bigg| \int_{S} \big( f_{k_\ell}(x,\lambda_{k_\ell} x) - f_\infty(x,\lambda_{k_\ell} x)\big) d\sigma(x)  \bigg| \\
			& \qquad + \bigg| \int_{S} f_{\infty}(x,\lambda_{k_\ell} x) d\sigma(x) -  \int_{S} \int_{\mathbb{T}^d}  f_{\infty}(x,y) dy d\sigma(x) \bigg| \\
			& \qquad + \bigg| \int_{S} \int_{\mathbb{T}^d}  f_{k_\ell}(x,y) dy d\sigma(x) - \int_{S} \int_{\mathbb{T}^d}  f_{\infty}(x,y) dy d\sigma(x) \bigg|\\
			& \le \bigg| \int_{S} f_{\infty}(x,\lambda_{k_\ell} x) d\sigma(x) -  \int_{S} \int_{\mathbb{T}^d}  f_{\infty}(x,y) dy d\sigma(x) \bigg| + 2\int_{S} \norm{f_{k_\ell}(x,\cdot) - f_\infty(x,\cdot)}_{C(\mathbb{T}^d)} d\sigma(x).
		\end{aligned}
	\end{equation*}
	As $k_\ell \to\infty$, the first term in the last inequality tends to zero by Theorem \ref{thm.mainOscInt}, while the second term tends to zero since $\{ f_{k_\ell} \}$ converge to $f_\infty$ in $L^1(S; C(\mathbb{T}^d) )$. This contradicts to (\ref{est.S.Contra}) and proves (\ref{eq.w.lambda}).
\end{proof}

As a straightforward application, we use Theorem \ref{thm.mainOscInt} to derive a homogenization theorem for harmonic functions with oscillating boundary data in Lipschitz domains satisfying the non-resonance condition.

\begin{theorem}\label{thm.Laplace.L2}
	Let $\Omega$ be a Lipschitz domain satisfying the non-resonance condition with respect to $\mathbb{Z}^d$. Assume  $g(x,y) : \partial\Omega\times \mathbb{T}^d \mapsto \mathbb{R}$ is 1-periodic and continuous in $y$ for each $x\in \partial\Omega$. Moreover, assume
	\begin{equation}\label{est.fxy.CL2}
	\int_{\partial\Omega} \norm{g(x,\cdot)}_{C(\mathbb{T}^d)}^2 d\sigma(x) < \infty.
	\end{equation}
	Let $u_\e$ be the solution of
	\begin{equation}\label{eq.Laplace.ue}
	\left\{
	\aligned
	- \Delta u_\e & = 0 \quad &\text{ in }& \Omega,\\
	u_\e(x) & =g(x,x/\e) \quad &\text{ on }& \partial\Omega.
	\endaligned
	\right.
	\end{equation}
	Then, $u_\e$ converges to $u_0$ pointwise (and in $L^2(\Omega)$), where $u_0$ is the solution of
	\begin{equation}\label{eq.Laplace.u0}
	\left\{
	\aligned
	-\Delta u_0 & = 0 \quad &\text{ in }& \Omega,\\
	u_0(x) & =\bar{g}(x) \quad &\text{ on }& \partial\Omega,
	\endaligned
	\right.
	\end{equation}
	where $\bar{g}(x) = \int_{\mathbb{T}^d} g(x,y) dy$.
\end{theorem}

\begin{proof}
	First of all, observe that $g(x,x/\e) \in L^2(\partial\Omega)$. By the $L^2$ theory for elliptic equations in Lipschitz domains (see, e.g., \cite{Dahl79,Verchota84} or \cite{KenigBook}), the Dirichlet problem (\ref{eq.Laplace.ue}) is solvable and
	\begin{equation*}
	\norm{\mathcal{N}(u_\e)}_{L^2(\partial\Omega)} \le \norm{g(\cdot,\cdot/\e)}_{L^2(\partial\Omega)} \le  \bigg( \int_{\partial\Omega} \norm{g(x,\cdot)}_{C(\mathbb{T}^d)}^2 d\sigma(x) \bigg)^{1/2},
	\end{equation*}
	where $\mathcal{N}(u_\e)$ is the non-tangential maximal function. Moreover, the harmonic measure $\omega^z(x)$, with $z\in \Omega$, is absolutely continuous with respect to $\sigma(x)$ and
	\begin{equation}\label{eq.def.k}
	k(\cdot,z):= \frac{d\omega^z}{d\sigma} \in L^2(\partial\Omega,d\sigma),
	\end{equation}
	and the solution $u_\e$ may be represented by
	\begin{equation*}
	u_\e(z) = \int_{\partial\Omega} g(x,x/\e) k(x,z) d\sigma(x), \qquad \text{for } z\in \Omega.
	\end{equation*}
	Let $f_z(x,y) = g(x,y) k(x,z)$ and note that by (\ref{est.fxy.CL2}) and (\ref{eq.def.k}),
	\begin{equation*}
	\begin{aligned}
	\int_{\partial\Omega} \norm{ f_z(x,\cdot)}_{C(\mathbb{T}^d)} d\sigma(x) & = \int_{\partial\Omega} \norm{ g(x,\cdot)}_{C(\mathbb{T}^d)} k(x,z) d\sigma(x) \\
	& \le \bigg( \int_{\partial\Omega} \norm{ g(x,\cdot)}_{C(\mathbb{T}^d)}^2 d\sigma(x) \bigg)^{1/2} \bigg( \int_{\partial\Omega} k^2(x,z) d\sigma(x) \bigg)^{1/2} \\
	& \le C_z
	\end{aligned}
	\end{equation*}
	It follows from Theorem \ref{thm.mainOscInt} that
	\begin{equation}\label{eq.ue2u0}
	\begin{aligned}
	\lim_{\e \to 0} u_\e(z) & = \lim_{\e \to 0} \int_{\partial\Omega} f_z(x,x/\e) d\sigma(x) \\
	& = \int_{\partial\Omega} \int_{\mathbb{T}^d} f_z(x,y) dy d\sigma(x) \\
	& = \int_{\partial\Omega} \bigg( \int_{\mathbb{T}^d} g(x,y) dy \bigg) k(x,z) d\sigma(x) \\
	& = u_0(z),
	\end{aligned}
	\end{equation}
	for each $z\in \Omega$, where $u_0$ is exactly the solution of (\ref{eq.Laplace.u0}). This gives the pointwise convergence of $u_\e$.
	
	Finally, we show that
	\begin{equation}\label{est.ue2u0.L2}
	\lim_{\e\to 0} \norm{u_\e - u_0}_{L^2(\Omega)} = 0.
	\end{equation}
	Actually, by the maximal principle, we have
	\begin{equation}\label{est.unif.bound}
	-\bar{u}(z) \le u_\e(z) \le \bar{u}(z),
	\end{equation}
	where $\bar{u}$ is the solution of
	\begin{equation}\label{eq.Delta.baru}
	\left\{
	\aligned
	-\Delta \bar{u} & = 0 \quad &\text{ in }& \Omega,\\
	\bar{u} (x) & =\norm{g(x,\cdot)}_{C(\mathbb{T}^d)} \quad &\text{ on }& \partial\Omega.
	\endaligned
	\right.
	\end{equation}
	By (\ref{est.fxy.CL2}) and the $L^2$ solvability of (\ref{eq.Delta.baru}) in Lipschitz domains (see \cite{Dahl79,Verchota84}), we know $\norm{\bar{u}}_{L^2(\Omega)} < \infty$. Therefore, (\ref{est.ue2u0.L2}) follows readily from the pointwise convergence (\ref{eq.ue2u0}), uniform estimate (\ref{est.unif.bound}) and the dominant convergence theorem.
\end{proof}

\section{Qualitative Homogenization}
In this section, we will apply Theorem \ref{thm.mainOscInt} to establish the qualitative homogenization theorem for the Robin boundary value problem (\ref{Robin}) in Lipschitz domains with a non-resonance condition. We begin with the definitions of correctors and the homogenized coefficient matrix \cite{BLP78,ShenBook}. For each $1\le j\le d, 1\le \beta\le m$, let $\chi = (\chi_j^\beta) = (\chi_j^{1\beta},\chi_j^{2\beta},\cdots,\chi_j^{m\beta})$ denote the correctors for $\mathcal{L}_\e$, which are 1-periodic functions satisfying the cell problem
\begin{equation*}
\left\{
\begin{aligned}
\mathcal{L}_1 (\chi^{\beta}_{j} + P_j^\beta) &= 0 \qquad  \text{ in } \mathbb{T}^d, \\
\int_{\mathbb{T}^d} \chi_j^\beta & = 0,
\end{aligned}
\right.
\end{equation*}
where $P_j^\beta(x) = x_je^\beta $ with $e^\beta$ being $\beta$th Cartesian basis in $\RM$. Recall that the homogenized matrix $\widehat{A} = (\hat{a}^{\alpha\beta}_{ij})$ is defined by
\begin{equation*}
\hat{a}^{\alpha\beta}_{ij} = \int_{\mathbb{T}^d} \bigg[ a^{\alpha\beta}_{ij} + a^{\alpha\gamma}_{ik} \frac{\partial}{\partial x_k} (\chi^{\gamma\beta}_j) \bigg] dx,
\end{equation*}
and the homogenized operator is given by $\mathcal{L}_0 = -\text{div}(\widehat{A} \nabla )$.

The following is the main theorem of this section.
\begin{theorem}\label{thm.Quali.Homo}
	Let $\{A_\ell, b_\ell\}$ be a sequence of coefficient matrices satisfying (\ref{cond.ellipticity}) -- (\ref{cond.continuity}). Moreover, we assume $\{b_\ell \}$ are equicontinuous. Assume $\Omega$ is a bounded Lipschitz domain whose boundary satisfies the non-resonance condition with respect to $\mathbb{Z}^d$, $\{ F_\ell \} \subset H^{-1}(\Omega;\RM)$ and $\{ g_\ell\} \subset H^{1/2}(\partial\Omega;\RM)$. Suppose that $\{ u_\ell\}$ are the weak solutions of
	\begin{equation}\label{eq.ul}
	\left\{
	\aligned
	-{\rm div}(A_\ell(x/\e_{\ell}) \nabla u_\ell) &= F_\ell &\quad& \text{ in } \Omega,\\
	n\cdot A_\ell(x/\e_\ell)\nabla u_\ell+b_\ell(x/{\e_\ell})u_{\ell} & =g_\ell &\quad &\text{ on } \partial\Omega,
	\endaligned
	\right.
	\end{equation}
	where $\e_\ell \to 0$ as $\ell \to \infty$, $u_\ell \in H^1(\Omega;\RM)$. We further assume, as $\ell \to \infty$,
	\begin{equation}
	\left\{
	\aligned
	F_\ell &\to F  &\quad & \text{in } H^{-1}(\Omega;\RM),\\
	g_\ell &\to g &\quad & \text{in } H^{-\frac{1}{2}}(\partial\Omega;\RM),\\
	u_\ell &\rightharpoonup u  &\quad &\text{weakly in } H^1(\Omega;\RM), \\
	\widehat{A_\ell} &\to A^0,& & \\
	\overline{b_\ell} &\to b^0,& &
	\endaligned
	\right.
	\end{equation}
	where $\widehat{A_\ell}$ denotes the effective coefficient matrix of $A_\ell$ and $\overline{b_\ell}$ denotes the effective diffusive matrix of $b_\ell$. Then
	\begin{equation}\label{eq.Al2A0}
	A_\ell(x/\e_{\ell})  \nabla u_\ell \rightharpoonup A^0 \nabla u \qquad \text{weakly in } L^2(\Omega;\mathbb{R}^{m\times d}),
	\end{equation}
	Moreover, $u$ is a weak solution of
	\begin{equation}
	\left\{
	\aligned
	-{\rm div}(A^0 \nabla u) &= F &\quad& \text{ in } \Omega,\\
	n\cdot A^0 \nabla u+b^0 u & =g &\quad &\text{ on } \partial\Omega.
	\endaligned
	\right.
	\end{equation}
\end{theorem}

\begin{proof}
	Thanks to the qualitative homogenization for the operator $\mathcal{L}_\e$ (see \cite[Theorem 2.3.2]{ShenBook}), our assumptions implies (\ref{eq.Al2A0}) and the equation $-{\rm div}(A^0 \nabla u) = F$ in $\Omega$. Thus, the key point is to verify the boundary condition. Note that the variational form of (\ref{eq.ul}) gives
	\begin{equation*}
	\int_{\Omega} A_\ell(x/\e_{\ell}) \nabla u_\ell \cdot \nabla \phi + \int_{\partial\Omega} b_\ell(x/\e_\ell) u_\ell \cdot \phi = \int_{\partial\Omega} g_\ell \cdot \phi + \int_{\Omega} F_\ell \cdot \phi,
	\end{equation*}
	for any $\phi\in C^\infty_0(\RD;\mathbb{R}^m)$. Clearly, by (\ref{eq.Al2A0}), we have
	\begin{equation*}
	\int_{\Omega} A_\ell(x/\e_{\ell}) \nabla u_\ell \cdot \nabla \phi \to \int_{\Omega} A^0 \nabla u \cdot \nabla \phi, \qquad \text{as } \ell \to \infty.
	\end{equation*}
	To verify the variational form of (\ref{eq.A0u}), it suffices to show
	\begin{equation}\label{eq.Bl2B0}
	\int_{\partial\Omega} b_\ell(x/\e_\ell) u_\ell \cdot \phi \to \int_{\partial\Omega} b^0 u \cdot \phi.
	\end{equation}
	
	First of all, by the equicontinuity and periodicity of $b_\ell$, as well as the convergence $\overline{b_\ell} \to b^0$, we know $\sup_{\ell} |b_\ell(y)| \le C.$
	Since $u_\ell$ converges to $u$ weakly in $H^1(\Omega;\RM)$, $u_\ell$ converges to $u$ strongly in $H^s(\Omega;\RM)$ with $\frac{1}{2}<s<1$. By the trace theorem,
	\begin{equation*}
	\norm{u_\ell - u}_{L^2(\partial\Omega)} \to 0,\quad \text{as } \ell \to \infty.
	\end{equation*}
	
	On the other hand, the Arzelà - Ascoli theorem implies that the set $\{ b_\ell(y)u(x)\cdot \phi(x) \}$ is compact in $L^1(S; C(\mathbb{T}^d) )$. Hence, by Theorem \ref{thm.rateXY},
	\begin{equation*}
	\bigg| \int_{\partial\Omega} b_\ell(x/\e_\ell) u \cdot \phi - \int_{\partial\Omega} \overline{b_\ell} u \cdot \phi \bigg| \le \omega(\e_\ell),
	\end{equation*}
	where $\omega(\e_\ell) \to 0$ as $\ell \to \infty$, and this rate depends at most on $\Omega, \norm{u\cdot \phi}_{H^1(\Omega)}$ and the modulus of equicontinuity of $\{b_\ell\}$. Consequently,
	\begin{equation*}
	\begin{aligned}
	&\bigg| \int_{\partial\Omega} b_\ell(x/\e_\ell) u_\e \cdot \phi - \int_{\partial\Omega} b^0 u \cdot \phi \bigg| \\
	& \qquad \le C\norm{u_\ell - u}_{L^2(\partial\Omega)} \norm{\phi}_{L^2(\partial\Omega)} + \omega(\e_\ell) + |\overline{b_\ell} - b^0| \norm{u}_{L^2(\partial\Omega)} \norm{\phi}_{L^2(\partial\Omega)},
	\end{aligned}
	\end{equation*}
	where the right-hand side converges to $0$ as $\ell \to \infty$. This proves (\ref{eq.Bl2B0}) and hence
	\begin{equation*}
	\int_{\Omega} A^0 \nabla u \cdot \nabla \phi + \int_{\partial\Omega} b^0 u \cdot \phi = \int_{\partial\Omega} g\cdot \phi + \int_{\Omega} f\cdot \phi,
	\end{equation*}
	which ends the proof.
\end{proof}

\begin{proof}[Proof of Theorem \ref{thm.main}]
	This is a simple corollary of Theorem \ref{thm.Quali.Homo}.
\end{proof}

\begin{remark}\label{rmk.layered}
Many materials in practice may have special microscopic structures beyond periodicity, such as layered materials (lamina) and directional materials (fiber, wood). These additional stronger structures may allows us to weaken the non-resonance condition that is indispensable for homogenization to take place on the boundary. Consider a material with a certain physical property,  described by a function $B$, in a fixed coordinate system. Assume that $B$ depends only on $k$ orthogonal directions $p_1, p_2,\cdots, p_k$, and remains constant along the rest orthogonal directions $p_{k+1}, \cdots, p_d$. Let $P = (p_1, p_2,\cdots, p_k)$ be a $d\times k$ matrix. Then, the previous assumption on $B$ is equivalent to the following structure equation
\begin{equation}\label{eq.B.layer}
	B(x) = B(PP^T x), \qquad \text{for any } x\in \RD.
\end{equation}
In particular, the layered and directional materials mentioned above are corresponding to the special cases $k = 1$ and $k=d-1$, respectively. We identify the structure equation for these two interesting cases:
\begin{equation}\label{eq.Layer.Direct}
\begin{aligned}
\text{Layered materials:} &\quad B(x) = B((p_1\otimes p_1) x);\\
\text{Directional materials:} &\quad B(x) = B((I - p_d\otimes p_d)x).
\end{aligned}
\end{equation}
Now, for such a material with structure (\ref{eq.B.layer}), the periodicity condition for $B$ will be imposed only on the linear subspace spanned by $\{p_1, p_2,\cdots, p_k \}$ (which is equivalent to $PP^T\RD$). As a result, we may redefine a weaker non-resonance condition as follows.
\begin{definition}\label{def.non-resonance.Gamma}
	Let $\Gamma \subset PP^T \RD$ be a periodic lattice. We say a closed Lipschitz surface $S$ satisfies the non-resonance condition with respect to $\Gamma$ if
	\begin{equation}\label{eq.non-resonance.Gamma}
		\sigma( \{x\in S: n(x) \text{ is well-defined and } n(x) \in \mathbb{R}\Gamma \} ) = 0.
	\end{equation}
\end{definition}
Definition \ref{def.non-resonance.Gamma} is weaker than Definition \ref{def.non-resonance} since we only need to verify (\ref{eq.non-resonance.Gamma}) for ``much less" directions. For example, we only need to verify a single direction $p_1$ (see (\ref{eq.Layer.Direct})) for a given layered material. Therefore, by a similar argument as before, the homogenization theorem may be established on a larger class of domains for directional or layered materials. The details will be omitted.

\end{remark}

\section{Auxiliary Neumann Problems}

This section and next one will be devoted to the quantitative convergence rates. As we have noticed, the main difficulty in Robin boundary value problems is the analysis of the integral in a form of
\begin{equation}\label{eq.osc.g}
	\int_{\partial\Omega} f(x/\e)\cdot \phi(x) d\sigma(x),
\end{equation}
where $f(y)$ is a $1$-periodic $\RM$-valued function. In this paper, we will use a ``duality approach'' to analyze (\ref{eq.osc.g}) quantitatively. Precisely, let $v_\e$ be the solution of the following auxiliary Neumann problem
\begin{equation}\label{eq.ve.Neumann}
\left\{
\aligned
-\text{div}(\widehat{A} \nabla v_\e) & =0 \quad &\text{ in }& \Omega,\\
n\cdot \widehat{A} \nabla v_\e & =f(x/\e) - M_\e \quad &\text{ on }& \partial\Omega.
\endaligned
\right.
\end{equation}
where $\widehat{A}$ is the (constant) homogenized matrix and $M_\e$ is a constant vector such that the compatibility condition is satisfied, i.e.,
\begin{equation*}
\int_{\partial\Omega} \big( f(x/\e) - M_\e \big)d\sigma(x) = 0.
\end{equation*}
Now, observing that by (\ref{eq.ve.Neumann}) and the integration by parts, one has
\begin{equation}\label{eq.osc.dual}
	\begin{aligned}
	\int_{\partial\Omega} f(x/\e)\cdot\phi(x) d\sigma(x) &= \int_{\Omega} \widehat{A}\nabla v_\e(x) \cdot \nabla \phi(x) dx + M_\e \int_{\partial\Omega} \phi(x)d\sigma(x)\\
	& = - \int_{\Omega} v_\e(x) \cdot \text{div}(\widehat{A}^*\nabla \phi(x)) dx + \int_{\partial\Omega} v_\e(x) (n\cdot\widehat{A}^* \nabla \phi(x))d\sigma(x) \\
	& \qquad+ M_\e \int_{\partial\Omega} \phi(x)d\sigma(x).
	\end{aligned}
\end{equation}
Thus, the estimate of (\ref{eq.osc.g}) is effectively reduced to the estimates of $v_\e$ or $\nabla v_\e$.

In this section, we focus on the estimates of $v_\e$ given by (\ref{eq.ve.Neumann}). Actually, (\ref{eq.ve.Neumann}) has been studied in \cite{ASS15} and we include their results (adapted to our situation) in the following theorem.
\begin{theorem}[\cite{ASS15}, Theorem 5.1 \& 5.3]\label{thm.ve.g}
	Let $\Omega$ be strictly convex and smooth. If $v_\e$ is the solution of (\ref{eq.ve.Neumann}) satisfying $\int_{\partial\Omega} v_\e d\sigma = 0$, then
	\begin{equation}\label{est.ve.Lp}
	\norm{v_\e }_{L^p(\Omega)} \le \left\{
	\aligned
	&C\e^{\frac{1}{p}} \qquad &\text{if } d =3, \\
	&C\e^{\frac{3}{2p}} \qquad &\text{if } d=4,\\
	&C\e^{\frac{2}{p}} |\ln \e|^{\frac{1}{p}} \qquad &\text{if } d\ge 5.
	\endaligned
	\right.
	\end{equation}
	and
	\begin{equation}\label{est.ve.W1q}
	\norm{\nabla v_\e}_{L^p(\Omega)} \le C_{\sigma} \e^{\frac{1}{p} - \sigma} , \qquad d\ge 3
	\end{equation}
	for any $1\le p<\infty$ and $\sigma>0$.
\end{theorem}

	The proof of Theorem \ref{thm.ve.g} is based on the integral representation for the solution and the estimates of oscillatory integrals on the boundary.
	Precisely, the solution of (\ref{eq.ve.Neumann}) may be given by
	\begin{equation}\label{eq.ve.N0}
	v_\e(x) = \int_{\partial\Omega} N_0(x,y)(f(y/\e) - M_\e) d\sigma(y).
	\end{equation}
	where $N_0(x,y)$ is the Neumann function of $-\text{div}(\widehat{A} \nabla)$ in $\Omega$. Recall that the Neumann function $N_0(x,y) = (N_0^{\alpha\beta}(x,y))$ is a matrix such that \cite{KLS14}
	\begin{equation*}
	\left\{
	\begin{aligned}
		&\mathcal{L}_0(N^\beta(\cdot, y)) = e^\beta \delta_y \qquad &\text{in } \Omega,\\
		&\frac{\partial}{\partial \nu_0}(N_0(\cdot,y) ) = -e^\beta |\partial\Omega|^{-1} \qquad &\text{on } \partial\Omega,\\
		&\int_{\partial\Omega} N_0(x,y) d\sigma(x) = 0&,
	\end{aligned}
	\right.
	\end{equation*}
	where $N^\beta_0 = (N_0^{1\beta}, N_0^{2\beta},\cdots, N_0^{m\beta})$ and $e^\beta = (0,\cdots, 1,\cdots, 0)$ is the $\beta$th Cartesian basis in $\RM$. Moreover, since $\Omega$ is smooth, we have
	\begin{equation}\label{est.N0xy}
	|\nabla_x^i \nabla_y^j N_0(x,y)| \le \frac{C_{i,j}}{|x-y|^{d-2+i+j}}.
	\end{equation}
	for any $i,j\ge 0$.
	
\begin{remark}
	If $x$ is an interior point of $\Omega$, the oscillatory integral theory for non-degenerate phase implies $|v_\e(x)| \le C\e^{\frac{d-1}{2}}$ (see \cite{ASS13}), which shows that the estimate in (\ref{est.ve.Lp}) should be sharp for $p = 1$ and $3\le d\le 5$. On the other hand, however, if $p>1$, the estimates in (\ref{est.ve.Lp}) do not give the sharp convergence rates. In particular, (\ref{est.ve.W1q}) and a Sobolev inequality lead to
	\begin{equation}\label{est.ve.Imp}
	\norm{v_\e}_{L^p(\Omega)} \le C\e^{\frac{1}{p}+\frac{1}{d} - \sigma},
	\end{equation}
	for any $\frac{d}{d-1}\le p\le \infty$ and any $\sigma>0$. Compared to (\ref{est.ve.Lp}), this is a better convergence rate for lower dimensions or larger $p$. In the most interesting case with $p = 2$, combining (\ref{est.ve.Lp}) (for $d\ge 4$) and (\ref{est.ve.Imp}) (for $d=3$), we arrive at
	\begin{equation}\label{est.ve.L2}
	\norm{v_\e }_{L^2(\Omega)} \le \left\{
	\aligned
	&C\e^{\frac{5}{6}-\sigma} \qquad &\text{if } d =3, \\
	&C\e^{\frac{3}{4}} \qquad &\text{if } d=4,\\
	&C\e |\ln \e|^{\frac{1}{2}} \qquad &\text{if } d\ge 5.
	\endaligned
	\right.
	\end{equation}
\end{remark}

In the following theorem, we make a key improvement on the estimate of $v_\e$.

\begin{theorem}
	Let $\Omega$ and $v_\e$ be the same as Theorem \ref{thm.ve.g}. Then, for $d\ge 3$,
	\begin{equation}\label{est.ve.infty}
	\norm{v_\e}_{L^\infty(\Omega)} \le C\e^{\frac{1}{2}}.
	\end{equation}
\end{theorem}
\begin{proof}
	The proof follows the idea of \cite{ASS15} and relies on the integral representation (\ref{eq.ve.N0}). By subtracting a constant, we may simply assume the mean of $f$ is zero. Also, since $f$ is $1$-periodic and smooth, by the Fourier series expansion, we may first consider $f(x) = g_k e^{2\pi i k\cdot x}$ with $k\in \mathbb{Z}^d\setminus\{ 0 \}$ and $f_k\in \RM$ is a constant vector. In this case
	\begin{equation*}
	M_\e = \frac{f_k}{|\partial\Omega|}\int_{\partial\Omega} e^{2\pi i\e^{-1} k\cdot x} d\sigma(x).
	\end{equation*}
	Since $\Omega$ is smooth and strictly convex, $|M_\e| \le C|f_k|\e^{\frac{d-1}{2}}$; see \cite{ASS13,Stein93}. 
	
	Hence, (\ref{eq.ve.N0}) gives
	\begin{equation}\label{eq.Sing.Osc}
	\begin{aligned}
	v_\e(x) & = \int_{\partial\Omega} N_0(x,y)f_k \big( e^{2\pi i \e^{-1} k\cdot y} - M_\e \big) d\sigma(y) \\
	& = \int_{\partial\Omega} N_0(x,y)f_k e^{2\pi i \e^{-1} k\cdot y}  d\sigma(y) + O(|f_k|\e^{\frac{d-1}{2}}).
	\end{aligned}
	\end{equation}
	To obtain the pointwise estimate of the integral in (\ref{eq.Sing.Osc}), it suffices to consider a typical case that $x\in \partial\Omega$. Fix such an $x\in \partial\Omega$. Let $\eta$ be a smooth cut-off function such that $\eta(y) = 1$ for $y\in B_{r_0}(x)$ and $\eta(y) = 0$ for $y\notin B_{2r_0}(x)$, where $r_0>0$ is an appropriately chosen radius (depending only on $d$ and $\Omega$) that $\partial\Omega$ may be localized near $x$. Moreover, $|\nabla^\ell \eta| \le C_\ell$. To deal with the singularity of the Neumann function at $x$, we introduce another cut-off function to rule out a small neighborhood of $x$. Let $\theta_\e(y) = 0$ in $B_{\e^{1/2}}(x)$ and $\theta_\e (y) = 1$ in $\RD\setminus B_{5\e^{1/2}}(x)$, and $|\nabla^\ell \theta_\e| \le C_\ell \e^{-\frac{\ell}{2}}$. Consequently, we have
	\begin{align*}
	\int_{\partial\Omega} N_0(x,y) f_k e^{2\pi i \e^{-1} k\cdot y}  d\sigma(y) &= \int_{\partial\Omega}\eta(y) \theta_\e(y)  N_0(x,y) f_k e^{2\pi i \e^{-1} k\cdot y}  d\sigma(y) \\
	& \qquad + \int_{\partial\Omega} (1 - \eta(y))  N_0(x,y) f_k e^{2\pi i \e^{-1} k\cdot y}  d\sigma(y) \\
	& \qquad + \int_{\partial\Omega} (1 - \theta_\e(y))  N_0(x,y) f_k e^{2\pi i \e^{-1} k\cdot y}  d\sigma(y) \\
	& = R_1 + R_2 + R_3.
	\end{align*}
	Note that $(1 - \eta(y))  N_0(x,y)$ has no singularity on the boundary and therefore, $|R_2| \le C|f_k| \e^{\frac{d-1}{2}}$. To estimate $R_3$, note that $1-\theta_\e$ is supported in $B_{5\e^{1/2}}(x)$ and thus
	\begin{equation*}
	|R_3| \le |f_k| \int_{\partial\Omega \cap B_{5\e^{1/2}}(x)} |N_0(x,y)| d\sigma(y) \le C|f_k| \e^{\frac{1}{2}},
	\end{equation*}
	where we have used (\ref{est.N0xy}) with $i=j=0$ in the last inequality. Hence, it suffices to estimate $R_1$. To do so, we first transform the surface integral to the usual one in $\mathbb{R}^{d-1}$. Precisely, we assume $z = Q^t(y-x)$ moves $x\in \partial\Omega$ to the origin and transform the tangent plane at $x$ to $z_d = 0$, where $Q\in \mathbb{R}^{d\times d}$ is an orthogonal matrix. As a result, $\partial\Omega \cap B(x,2r_0)$ is transformed to the local graph $z_d = \phi(z')$ which satisfies $\phi(0) = 0$ and $\nabla \phi(0) = 0$. It follows that
	\begin{equation*}
	\begin{aligned}
	R_1 &= \int_{\{ |z'|< 2r_0\} } \eta(x+Qz) \theta_\e(x+Qz) N_0(x,x+Qz) f_k e^{2\pi i \e^{-1} k\cdot (x+Qz)} \sqrt{1+|\nabla \phi(z')|^2} dz' \\
	& = K_x \int_{\{ |z'|< 2r_0\} } \tilde{\eta}(z) \tilde{\theta}_\e(z) \widetilde{N}_0(z)f_k  e^{2\pi i \e^{-1} Q^t k\cdot z} dz',
	\end{aligned}
	\end{equation*}
	where $K_x = e^{2\pi i \e^{-1} k\cdot x}, \tilde{\eta}(z) =\eta(x+Qz) \sqrt{1+|\nabla \phi(z')|^2},\ \tilde{\theta}_\e(z) = \theta_\e(x+Qz)$ and $\widetilde{N}_0(z) = N_0(x,x+Qz)$. By our construction, one knows $\tilde{\eta}(z)$ is a smooth cut-off function supported in $B_{2r_0}(0)$; $\tilde{\theta}_\e(z)$ is a smooth cut-off function vanishing in $B_{\e^{1/2}}(0)$ and $|\nabla^\ell \tilde{\theta}_\e| \le C\e^{-\frac{\ell}{2}}$. Moreover, (\ref{est.N0xy}) implies
	\begin{equation}\label{est.tN0}
	|\nabla^\ell \widetilde{N}_0(z)| \le \frac{C}{|z'|^{d-2+\ell}}, \qquad \ell \ge 0.
	\end{equation}
	
	Now, let $n = (n',n_d) = \frac{Q^t k}{|Q^t k|}$. Using $z_d = \phi(z')$, we have
	\begin{equation}\label{eq.R1}
	R_1 = K_x \int_{\{ |z'|< 2r_0\} } \tilde{\eta}(z) \tilde{\theta}_\e(z) \widetilde{N}_0(z)f_k  e^{2\pi i \e^{-1}|k| (n'\cdot z' + n_d \phi(z'))} dz'.
	\end{equation}
	Let $F(z') = n'\cdot z' + n_d \phi(z')$. We need to discuss two cases separately.
	
	{\bf Case 1:} $|n'|>C_0|n_d|$ for some $C_0>0$. In this case, there exists $n_j$ with some $1\le j\le d-1$ so that $|n_j|\ge |n'|/(d-1) > C_0|n_d|/(d-1)$. Thus, if $C_0$ is large enough,
	\begin{equation}\label{est.D_jF}
	\Big|\frac{\partial}{\partial z_j} F(z')\Big| = \Big|n_j + n_d \frac{\partial}{\partial z_j} \phi(z') \Big| \ge \frac{|n_j|}{2} \ge \frac{1}{4(d-1)}.
	\end{equation}
	for any $z'$ with $|z'|< 2r_0$. Then, by an integration by parts, (\ref{eq.R1}) gives
	\begin{equation*}
	R_1  = -K_x \int_{\{ |z'|< 2r_0\} } \frac{\partial}{\partial z_j} \bigg[ \tilde{\eta}(z) \tilde{\theta}_\e(z) \widetilde{N}_0(z)f_k \Big( 2\pi i \e^{-1}|k| \frac{\partial}{\partial z_j} F(z') \Big)^{-1} \bigg] e^{2\pi i \e^{-1}|k| F(z') } dz'.
	\end{equation*}
	It follows from (\ref{est.D_jF}) and (\ref{est.tN0}) that
	\begin{equation*}
	\begin{aligned}
	|R_1| &\le \frac{C|f_k| \e}{|k|}\int_{\{ \e^{1/2} < |z'| < 2r_0\}} |\widetilde{N}_0(z)|dz' + \frac{C|f_k|\e^{\frac{1}{2}}}{|k|}\int_{ \{\e^{1/2} < |z'| < 5\e^{1/2} \} } |\widetilde{N}_0(z)|dz' \\
	& \qquad + \frac{C|f_k|\e}{|k|}\int_{\{ \e^{1/2} < |z'| < 2r_0\} } |\nabla \widetilde{N}_0(z)|dz' \\
	& \le \frac{C|f_k|\cdot \e |\ln \e|}{|k|},
	\end{aligned}
	\end{equation*}
	which gives a desired bound for the first case.
	
	{\bf Case 2:} $|n'|\le C_0 |n_d|$ for some $C_0$. Note that this implies $|n_d| \ge 1/\sqrt{1+C_0^2}$. It turns out that $\nabla^2 F(z') = n_d \nabla^2 \phi(z')$ is non-degenerate in $\{ |z'| < 2r_0 \}$. Therefore, there exists at most one point $w' \in \{ |z'|<2r_0 \}$ such that $\nabla F(w') = 0$. Without loss of generality, we may assume $\nabla^2 F(w')$ is diagonal with a minimum eigenvalue $\lambda_0>0$ depending only on the domain $\Omega$ (This can be done by making a rotation for $z'$). It follows that there exists $c_0>0$ such that if $|z' - w'|<c_0$,
	\begin{equation*}
	\frac{\partial^2 F(z')}{\partial {z}_j^2} \ge \frac{99}{100}\lambda_0, \qquad \frac{\partial^2 F(z')}{\partial z_i \partial z_j } \le \frac{1}{100d}\lambda_0, \qquad \text{if } i\neq j.
	\end{equation*}
	This implies that
	\begin{equation}\label{est.Fzj}
	\Big|\frac{\partial F(z')}{\partial z_j}\Big| \ge c|z_j|,
	\end{equation}
	if $|z'|<2r_0$ and $z'-w'\in \mathcal{C}_j$, where $\mathcal{C}_j$ is a cone defined by
	\begin{equation}\label{eq.Cj}
	\mathcal{C}_j: = \Big\{z'\in \mathbb{R}^{d-1}: |z_j|> \frac{1}{2\sqrt{d-1}} |z|  \Big\}.
	\end{equation}
	For each $j = 1,2,\cdots, d-1$, there exists smooth $\rho_j$ supported in $\mathcal{C}_j \setminus B_{\e^{1/2}}(0)$ such that
	\begin{equation*}
	\sum_{j = 1}^{d-1} \rho_j(z') = 1,\qquad \text{for any } z'\in \mathbb{R}^{d-1}\setminus B_{2 \e^{1/2}}(0),
	\end{equation*}
	and $|\nabla^\ell \rho_j(z')|\le C\e^{-\frac{\ell}{2}}$. Let $\rho_0 = 1 - \sum_{j=1}^{d-1} \rho_j$.
	
	To proceed, we write
	\begin{equation*}
	R_1 = \sum_{j =0}^{d-1} K_x \int_{\{ |z'|< 2r_0\} } \rho_j(z'-w') \tilde{\eta}(z) \tilde{\theta}_\e(z) \widetilde{N}_0(z)f_k  e^{2\pi i \e^{-1}|k| F(z')} dz'.
	\end{equation*}
	Let $R_{1}^j$ be the integral above with $j\ge 0$. Since $\rho_0(z'-w')$ is supported in $B_{2\e^{1/2}}(w')$, it is easy to see
	\begin{equation*}
	|R_1^0| \le C|f_k| \int_{B_{2\e^{1/2}}(w')} |\widetilde{N}_0(z)|dz' \le C|f_k| \e^{\frac{1}{2}}
	\end{equation*}
	For $1\le j \le d-1$, applying the integration by parts twice, one has
	\begin{equation*}
	\begin{aligned}
	R_1^j 
	=& K_x (2\pi i \e^{-1}|k| )^{-2} \int_{\{ |z'|< 2r_0\} } e^{2\pi i \e^{-1}|k| F(z')} \\
	&\qquad \times \frac{\partial}{\partial z_j}\bigg[ \Big( \frac{\partial F(z')}{\partial z_j}\Big)^{-1} \frac{\partial}{\partial z_j}\bigg[ \Big( \frac{\partial F(z')}{\partial z_j}\Big)^{-1} \rho_j(z'-w') \tilde{\eta}(z) \tilde{\theta}_\e(z) \widetilde{N}_0(z)f_k \bigg]\bigg]   dz'
	\end{aligned}
	\end{equation*}
	It follows from (\ref{est.Fzj}) and (\ref{eq.Cj}) that
	\begin{equation*}
	\begin{aligned}
	|R_1^j|  &\le \frac{C|f_k| \e^2}{|k|^2} \int_{ \{ \e^{\frac{1}{2}} \le |z'|<2r_0,\ |z'-w'|\ge \e^{\frac{1}{2}} \}} \frac{|\widetilde{N}_0(z')|}{|z'-w'|^4}  dz' \\
	& \qquad + \frac{C|f_k| \e^2}{|k|^2} \int_{ \{ \e^{\frac{1}{2}} \le |z'|<2r_0,\ |z'-w'|\ge \e^{\frac{1}{2}} \}} \frac{\e^{-\frac{1}{2}} |\widetilde{N}_0(z')| + |\nabla \widetilde{N}_0(z')|}{|z'-w'|^3} dz' \\
	& \qquad + \frac{C|f_k| \e^2}{|k|^2} \int_{ \{ \e^{\frac{1}{2}} \le |z'|<2r_0,\ |z'-w'|\ge \e^{\frac{1}{2}} \}} \frac{\e^{-1} |\widetilde{N}_0(z')| + \e^{-\frac{1}{2}}|\nabla \widetilde{N}_0(z')| +  |\nabla^2 \widetilde{N}_0(z')|}{|z'-w'|^2} dz'
	\end{aligned}
	\end{equation*}
	By using the (\ref{est.tN0}), we see that each term above is bounded by $C|k|^{-2} |f_k| \e^{ \frac{1}{2}}$. This gives a desired estimate of $R_1$ in the second case.
	
	Finally, combining the estimates for $R_1, R_2$ and $R_3$, as well as the estimate of $M_\e$, we obtain $|v_\e(x)| \le C|f_k| \e^{\frac{1}{2}}$ in the particular situation that $f(x) = f_k e^{2\pi i k\cdot x}$. For general smooth $1$-periodic function $f = \sum_{k\neq 0} f_k e^{2\pi i k\cdot x}$, the previous estimate leads to
	\begin{equation*}
	|v_\e(x)| \le C\e^{\frac{1}{2}}\sum_{k\neq 0} |f_k| \le C\e^{\frac{1}{2}},
	\end{equation*}
	where the last inequality is valid because $f$ is smooth.
\end{proof}

\begin{theorem}\label{thm.ve.improve}
	Let $\Omega$ and $v_\e$ be the same as Theorem \ref{thm.ve.g}. Then,
	\begin{equation}\label{est.veLp.improved}
	\norm{v_\e }_{L^p(\Omega)} \le \left\{
	\aligned
	&C\e^{\min \{ \frac{1}{2} + \frac{3}{4p}, 1\}- \sigma} \qquad &\text{if } d =3, \\
	&C\e^{ \frac{1}{2} + \frac{1}{p} } \qquad &\text{if } d=4,\\
	&C\e^{ \frac{1}{2} + \frac{3}{2p} } |\ln \e|^{\frac{1}{p}} \qquad &\text{if } d\ge 5,
	\endaligned
	\right.
	\end{equation}
	where $1\le p\le \infty$ and $\sigma>0$ may be arbitrarily small ($\sigma = 0$ if $p = 1$ or $\infty$).
\end{theorem}
\begin{proof}
	If $d=3$, the estimate $\norm{\nabla v_\e}_{L^1(\Omega)} \le C\e^{1-\sigma}$ and the Sobolev inequality implies that
	\begin{equation}\label{est.L1.5}
	\norm{v_\e}_{L^{\frac{3}{2}}(\Omega)} \le C\e^{1-\sigma}.
	\end{equation}
	Thus, the same estimate holds for $\norm{v_\e }_{L^p(\Omega)}$ for any $1< p<\frac{3}{2}$. If $\frac{3}{2}<p<\infty$, interpolating the estimate (\ref{est.L1.5}) and (\ref{est.ve.infty}) gives
	\begin{equation*}
	\norm{v_\e}_{L^p(\Omega)} \le C\e^{\frac{1}{2}+\frac{3}{4p}}.
	\end{equation*}
	
	For $d = 4$ or $d\ge 5$, an interpolation between (\ref{est.ve.Lp}) with $p = 1$ and (\ref{est.ve.infty}) gives the desired estimates.
\end{proof}

\begin{remark}
	For the special case $p = 2$, the above theorem gives
	\begin{equation}\label{est.veL2.sharp}
	\norm{v_\e }_{L^2(\Omega)} \le \left\{
	\aligned
	&C\e^{\frac{7}{8}- \sigma} \qquad &\text{if } d =3, \\
	&C\e \qquad &\text{if } d= 4,\\
	&C\e^{\frac{5}{4}}|\ln \e|^{\frac{1}{2}} \qquad &\text{if } d\ge 5,
	\endaligned
	\right.
	\end{equation}
	which is a significant improvement of (\ref{est.ve.L2}). In this paper, the expected convergence rates for the Robin problem (\ref{Robin}) cannot be better than (\ref{est.veL2.sharp}) in any dimensions. 
\end{remark}

The following estimate for the trace of $v_\e$ will also be useful.
\begin{lemma}
	Let $\Omega$ and $v_\e$ be the same as Theorem \ref{thm.ve.g}. Then for any $1\le p\le \infty$,
	\begin{equation}\label{est.ve34d}
		\norm{v_\e}_{L^p(\partial\Omega)} \le \left\{ 
		\begin{aligned}
		 &C\e^{\frac{1}{2}+\frac{1}{2p}-\sigma}, \quad &\text{if }& d = 3,4,\\
		 &C\e^{\min \{ \frac{1}{2}+\frac{1}{p}, 1\} -\sigma}, \quad &\text{if }& d\ge 5.
		\end{aligned}
		\right.
	\end{equation}
\end{lemma}
\begin{proof}
	The lemma is a straightforward of Theorem \ref{thm.ve.improve} and (\ref{est.ve.W1q}). First of all, for any $g \in C_0^\infty(\RD)$, one has
	\begin{equation*}
	\norm{g}_{L^1(\partial\Omega)} \le C\norm{g}_{L^1(\Omega)} + C\norm{\nabla g}_{L^1(\Omega)}.
	\end{equation*}
	Applying $g = v_\e$ to the above inequality and using Theorem \ref{thm.ve.g}, we have $\norm{v_\e}_{L^1(\partial\Omega)} \le C\e^{1-\sigma}$. Then the general estimates of $\norm{v_\e}_{L^p(\Omega)}$ for $d = 3$ or $4$ follows from an interpolation with (\ref{est.ve.infty}).

	To handle the case $d\ge 5$, we claim that for any $q\in (1,\infty)$,
	\begin{equation}\label{est.trace.f}
		\norm{g}_{L^2(\partial\Omega)} \le C\norm{g}_{L^2(\Omega)} + C\norm{g}_{L^{q}(\Omega)}^{\frac{1}{2}}  \norm{\nabla g}_{L^{q'}(\Omega)}^{\frac{1}{2}}.
	\end{equation}
	It turns out that the above trace estimate implies $\norm{v_\e}_{L^2(\partial\Omega)} \le C\e^{1-\sigma}$ for $d\ge 5$. Actually, applying (\ref{est.trace.f}) to $v_\e$ and using Theorem \ref{thm.ve.g}, we obtain that
	\begin{equation*}
		\begin{aligned}
			\norm{v_\e}_{L^2(\partial\Omega)} & \le C\norm{v_\e}_{L^2(\Omega)} + C\norm{v_\e}_{L^{q}(\Omega)}^{\frac{1}{2}}  \norm{\nabla v_\e}_{L^{q'}(\Omega)}^{\frac{1}{2}} \\
			& \le C\e|\ln \e|^{\frac{1}{2}} + C \e^{\frac{1}{q}} |\ln \e|^{\frac{1}{2q}} \e^{\frac{1}{2q'}} \\
			& \le C \e^{\frac{1}{q}}.
		\end{aligned}
	\end{equation*}
	The desired estimate follows by choosing $q$ sufficiently close to $1$ such that $\frac{1}{q} = 1-\sigma$. Finally, the estimate of $\norm{v_\e}_{L^p(\partial\Omega)}$ for $p>2$ follows from the interpolation with (\ref{est.ve.infty}).
\end{proof}

\section{Convergence Rates}
This section is devoted to the convergence rates for the Robin problem (\ref{Robin}) under perfect conditions. Precisely, we assume that $b \in C^\infty(\mathbb{T}^d;\mathbb{R}^{m\times m})$ and $\Omega$ is a smooth and strictly convex domain. Since we do not impose any regularity on the matrix $A$, as usual, we need a smoothing operator to deal with this situation. Let $\phi\in C^\infty_0(B_1(0))$ such that $\int \phi = 1$. Define $\phi_\e(x) = \e^{-d} \phi(\e^{-1} x)$ and
\begin{equation*}
S_\e(f)(x) = \int_{\RD} \phi_\e(y) f(x-y) dy.
\end{equation*}
Many useful properties of operator $S_\e$ may be found in, e.g., \cite{Shen17,ShenBook}.

Assume $u_\e$ and $u_0$ are the solutions of (\ref{Robin}) and the corresponding homogenized system (\ref{eq.A0u}). First of all, we construct the first-order expansion and establish the convergence rate in $H^1$. Let $\eta_\e$ be a cut-off function such that $\eta_\e = 1$ in $\{x\in \Omega: \text{dist}(x,\partial\Omega) > 2\e \}$, $\eta = 0$ in $\{x\in \Omega: \text{dist}(x,\partial\Omega) < \e \}$ and $|\nabla \eta_\e| \le C\e^{-1}$. Define the error of the first-order approximation by
\begin{equation}\label{eq.we}
w_\e = u_\e - u_0 - \e \chi(x/\e) S_\e(  \eta_\e\nabla u_0) .
\end{equation}
By the variational equations for $u_\e$ and $u_0$, for any $\varphi\in H^1(\Omega;\RM)$, one has
	\begin{equation*}
	\int_{\Omega} A_\e \nabla u_\e \cdot \nabla \varphi + \int_{\partial\Omega} b_\e u_\e \cdot \varphi = \int_{\Omega} \widehat{A} \nabla u_0 \cdot \nabla \varphi + \int_{\partial\Omega} \overline{b} u_0 \cdot \varphi,
	\end{equation*}
	where $A_\e(x) = A(x/\e)$. Then, a straightforward computation shows that $w_\e$ satisfies
	\begin{equation}\label{eq.splitInt}
	\begin{aligned}
	\int_{\Omega} A_\e \nabla w_\e \cdot \nabla \varphi + \int_{\partial\Omega} b_\e w_\e \cdot \varphi
	&  = \int_{\Omega} (\widehat{A} - A_\e - A_\e \nabla \chi(x/\e)) S_\e (\eta_\e \nabla u_0) \cdot \nabla \varphi \\
	& \qquad + \int_{\Omega} (\widehat{A} - A_\e) (\nabla u_0 -  S_\e( \eta_\e \nabla u_0 ))\cdot \nabla \varphi \\
	& \qquad + \int_{\Omega} \e A_\e \chi(x/\e) \nabla S_\e (\eta_\e \nabla u_0) \cdot \nabla \varphi \\
	& \qquad + \int_{\partial\Omega} (\overline{b} - b(x/\e)) u_0 \cdot \varphi \\
	& = I_1 + I_2 + I_3 + I_4.
	\end{aligned}
	\end{equation}
	The estimates for $I_1, I_2$ and $I_3$ follow from the same argument as Dirichlet or Neumann problems; see, e.g., \cite[Lemma 3.5]{SZ17}. Indeed, one can prove that
	\begin{equation}\label{est.I123}
	|I_1| + |I_2| + |I_3| \le C\e^{\frac{1}{2}} \norm{u_0}_{H^2(\Omega)} \big( \e^{\frac{1}{2}} \norm{\nabla \varphi}_{L^2(\Omega)} + \norm{\nabla \varphi}_{L^2(\Omega_{2\e})} \big),
	\end{equation}
	where $\Omega_{t} = \{ x\in \Omega: \text{dist}(x,\partial\Omega)<t \}$.
	
	The main difficulty caused by the Robin boundary condition is the estimate of $I_4$. To handle this term, we write
	\begin{equation}\label{eq.I4}
	I_4 = \int_{\partial\Omega} (\overline{b} - b(x/\e) - M_\e) u_0 \cdot \varphi +  \int_{\partial\Omega} M_\e u_0 \cdot \varphi,
	\end{equation}
	where $M_\e \in \mathbb{R}^{m\times m}$ is a constant matrix given by
	\begin{equation*}
	M_\e = \dashint_{\partial\Omega} (\overline{b} - b(x/\e) ) d\sigma(x).
	\end{equation*}
	Since $\Omega$ is strictly convex and smooth, we may employ the classical results in oscillatory integral theory (with non-degenerate phases) to obtain
	\begin{equation}\label{est.Me}
	|M_\e| \le C\e^{\frac{d-1}{2}},
	\end{equation}
	where the constant $C$ depends only on $b$ and $\Omega$. It follows that
	\begin{equation}\label{est.Meu0}
	\bigg| \int_{\partial\Omega} M_\e u_0 \cdot \varphi \bigg| \le C\e \norm{u_0}_{L^2(\partial\Omega)} \norm{\varphi}_{L^2(\partial\Omega)}.
	\end{equation}
	for $d\ge 3$.
	
	Thus, it is sufficient to consider the first integral of (\ref{eq.I4}), namely,
	\begin{equation*}
	I_{\text{osc}} = \int_{\partial\Omega} (\overline{b} - b(x/\e) - M_\e) u_0 \cdot \varphi.
	\end{equation*}
	As we have mentioned, this can be done by a ``duality approach" via a Neumann problem. For each $\beta$ with $1\le \beta \le m$, let $b^{\beta} = (b^{1\beta}, b^{2\beta}, \cdots, b^{m\beta})$. Let $v_\e^\beta$ be the solution of
	\begin{equation}\label{eq.ve.b}
	\left\{
	\aligned
	-\text{div}(\widehat{A} \nabla v_\e^\beta) & =0 \quad &\text{ in }& \Omega,\\
	n\cdot \widehat{A} \nabla v_\e^\beta & =\overline{b}^\beta - b^\beta(x/\e) - M_\e^\beta \quad &\text{ on }& \partial\Omega.
	\endaligned
	\right.
	\end{equation}
	where $\bar{b}^\beta$ and $M_\e^\beta$ are the $\beta$th column of $\bar{b}$ and $M_\e$. Under the assumption that $\Omega$ is strictly convex and smooth, all the estimates in the previous section are valid for $v_\e = (v_\e^1,v_\e^2,\cdots, v_\e^m)$.
	
	\begin{lemma}\label{lem.Iosc.H1}
		Let $\varphi \in H^1(\Omega;\RM)$.
		
		(i) If $u_0\in W^{1,d}(\Omega;\RM)$, then
		\begin{equation*}
		|I_{\rm{osc}}| \le C\e^{\frac{1}{2} - \sigma}  \norm{u_0}_{W^{1,d}(\Omega)}  \norm{ \varphi}_{H^1(\Omega)}.
		\end{equation*}
		
		(ii) If $u_0\in H^2(\Omega;\RM)$, then
		\begin{equation}
		|I_{\rm{osc}}| \le \left\{
		\aligned
		&C\e^{\frac{2}{3}- \sigma}\norm{u_0}_{H^2(\Omega)} \norm{\varphi}_{H^1(\Omega)} \qquad &\text{if } d =3, \\
		&C\e^{\frac{1}{2}-\sigma}\norm{u_0}_{H^2(\Omega)} \norm{\varphi}_{H^1(\Omega)} \qquad &\text{if } d = 4,\\
		&C\e^{\frac{1}{4} + \frac{3}{4(d-1)} - \sigma }\norm{u_0}_{H^2(\Omega)} \norm{\varphi}_{H^1(\Omega)} \qquad &\text{if } d \ge 5.
		\endaligned
		\right.
		\end{equation}
	\end{lemma}

	\begin{proof}
(i) By the construction of $v_\e$ in (\ref{eq.ve.b}) and the integration by parts, one has
	\begin{equation*}
	\begin{aligned}
	|I_{\rm{osc}}| & = \bigg| \int_{\partial\Omega} (n\cdot \widehat{A}\nabla v_\e) u_0 \cdot \varphi \bigg| \\
	&  = \bigg| \int_{\Omega} \nabla v_\e \nabla u_0  \cdot \varphi + \int_{\Omega} \nabla v_\e u_0 \cdot \nabla\varphi \bigg|\\
	&  \le C\big(  \norm{\nabla v_\e}_{L^{2}(\Omega)} \norm{\nabla u_0}_{L^d(\Omega)} + \norm{\nabla v_\e}_{L^{2+\sigma}(\Omega)} \norm{ u_0}_{L^{\frac{4}{\sigma}+2}(\Omega)}\big) \norm{\varphi}_{H^1(\Omega)} \\
	& \le C\e^{\frac{1}{2} - \sigma} \norm{u_0}_{W^{1,d}(\Omega)} \norm{ \varphi}_{H^1(\Omega)},
	\end{aligned}
	\end{equation*}
	where we have used the Sobolev inequality and (\ref{est.ve.W1q}) with appropriate $\sigma$ in the last two inequalities.
	
	(ii) If $u_0\in H^2(\Omega;\RM)$, the argument in (i) only gives $|I_{\text{osc}}| \le C\e^{\frac{2}{d}-\sigma} \norm{u_0}_{H^2(\Omega)} \norm{\varphi}_{H^1(\Omega)}$, which gives the desired estimates for $d = 3,4$. In the following, we use a more careful argument to improve this estimate for $d\ge 5$. We first assume $\varphi \in H^1(\partial\Omega;\RM)$. Then the Sobolev inequality implies $u_0 \varphi \in W^{1,p_0}(\partial\Omega;\RM)$, where $\frac{1}{p_0} = 1-\frac{3}{2(d-1)}$, and
	\begin{equation}\label{est.u0varphi.Lp}
	\norm{u_0 \varphi}_{W^{1,p_0}(\partial\Omega)} \le C\norm{u_0}_{H^{\frac{3}{2}}(\partial\Omega)} \norm{\varphi}_{H^1(\partial\Omega)}.
	\end{equation}
	
	Now, we construct $\Phi = (\Phi^\beta)$ as the solution of
	\begin{equation}\label{eq.Phi}
	\left\{
	\aligned
	-\text{div}(\widehat{A}^* \nabla \Phi^\beta) & =0 \quad &\text{ in }& \Omega,\\
	\Phi^\beta & =u_0^\beta \varphi \quad &\text{ on }& \partial\Omega.
	\endaligned
	\right.
	\end{equation}
	By the regularity theory of (\ref{eq.Phi}) with $W^{1,p}$ Dirichlet boundary value in smooth domains, we have
	\begin{equation}\label{est.Phi.Lp}
	\norm{\nabla \Phi}_{L^{p_0}(\partial\Omega)} \le \norm{\mathcal{N}(\nabla \Phi)}_{L^{p_0}(\partial\Omega)} \le C\norm{\nabla_{\text{tan}} (u_0 \varphi)}_{L^{p_0}(\partial\Omega)}.
	\end{equation}
	where $\mathcal{N}(\nabla \Phi)$ is the non-tangential maximal function defined by
	\begin{equation*}
	\mathcal{N}(f)(Q) = \sup \{ |f(x)|: x\in \Omega, |x-Q|< (1+\alpha) \text{dist}(x,\partial\Omega) \},
	\end{equation*}
	where $\alpha>0$ is a fixed constant. Particularly, (\ref{est.Phi.Lp}) implies that the Dirichlet-to-Neumann map $\Phi^\beta \mapsto n\cdot \widehat{A}^* \nabla \Phi^\beta$ is bounded from $W^{1,p_0}(\partial\Omega;\RM)$ to $L^{p_0}(\partial\Omega;\RM)$. Also, it follows from the Green's second identity that
	\begin{equation*}
	I_{\text{osc}} = \int_{\partial\Omega} (n\cdot \widehat{A}\nabla v_\e) \cdot \Phi = \int_{\partial\Omega}  v_\e \cdot (n\cdot \widehat{A}^* \nabla \Phi).
	\end{equation*}
	As a result, if $d\ge 5$, (\ref{est.Phi.Lp}) and (\ref{est.u0varphi.Lp}) imply that
	\begin{equation}\label{est.Iosc.H1}
	\begin{aligned}
	|I_{\text{osc}}| &\le C\norm{v_\e}_{L^{p_0'}(\partial\Omega)} \norm{\nabla \Phi}_{L^{p_0}(\partial\Omega)} \\
	& \le C\e^{\frac{1}{2} + \frac{1}{p_0'} -\sigma} \norm{u_0}_{H^{\frac{3}{2}}(\partial\Omega)} \norm{\varphi}_{H^1(\partial\Omega)} \\
	& \le C\e^{\frac{1}{2} + \frac{3}{2(d-1)} -\sigma} \norm{u_0}_{H^{\frac{3}{2}}(\partial\Omega)} \norm{\varphi}_{H^1(\partial\Omega)},
	\end{aligned}
	\end{equation}
	where we have used (\ref{est.ve34d}) and the fact that $\frac{1}{p_0'} = \frac{3}{2(d-1)}$.
	
	Clearly, note that
	\begin{equation}\label{est.Iosc.L2}
	|I_{\text{osc}}| \le C \norm{u_0}_{L^2(\partial\Omega)} \norm{\varphi}_{L^2(\partial\Omega)}.
	\end{equation}
	Thus, an interpolation between (\ref{est.Iosc.H1}) and (\ref{est.Iosc.L2}) gives
	\begin{equation*}
	|I_\text{osc}| \le C\e^{\frac{1}{4} + \frac{3}{4(d-1)}} \norm{u_0}_{H^{\frac{3}{2}}(\partial\Omega)} \norm{\varphi}_{H^\frac{1}{2}(\partial\Omega)} \le C\e^{\frac{1}{4} + \frac{3}{4(d-1)}} \norm{u_0}_{H^2(\Omega)} \norm{\varphi}_{H^1(\Omega)}.
	\end{equation*}
	This ends the proof.
\end{proof}

\begin{theorem}\label{thm.we.H1}
	Let $u_\e$ and $u_0$ be the same as before. Then
	
	(i) If $u_0\in H^2(\Omega;\RM)$,
	\begin{equation*}
	\norm{w_\e}_{H^1(\Omega)} \le \left\{
	\aligned
	&C\e^{\frac{1}{2}}\norm{u_0}_{H^2(\Omega)} \qquad &\text{if } d =3, \\
	&C\e^{\frac{1}{2}-\sigma}\norm{u_0}_{H^2(\Omega)}  \qquad &\text{if } d = 4,\\
	&C\e^{\frac{1}{4} + \frac{3}{4(d-1)} - \sigma }\norm{u_0}_{H^2(\Omega)} \qquad &\text{if } d \ge 5.
	\endaligned
	\right.
	\end{equation*}
	
	(ii) If $d\ge 5$ and $u_0\in H^2 \cap W^{1,d}(\Omega;\RM)$,
	\begin{equation}\label{est.weH1.5d}
	\norm{w_\e}_{H^1(\Omega)} \le C\e^{\frac{1}{2} -\sigma} \big( \norm{u_0}_{W^{1,d}(\Omega)} + \norm{u_0}_{H^2(\Omega)} \big).
	\end{equation}
\end{theorem}
\begin{proof}
	(i) In view of (\ref{eq.splitInt}), (\ref{est.I123}), (\ref{est.Meu0}) and Lemma \ref{lem.Iosc.H1} (ii), one has
		\begin{equation*}
			\begin{aligned}
			\bigg| \int_{\Omega} A_\e \nabla w_\e \cdot \nabla \varphi + \int_{\partial\Omega} b_\e w_\e \cdot \varphi \bigg| &\le C\norm{u_0}_{H^2(\Omega)} \big( \e \norm{\varphi}_{H^1(\Omega)} + \e^{\frac{1}{2}} \norm{\nabla \varphi}_{L^2(\Omega_{2\e})} \big) \\
			& \qquad + \left\{
			\aligned
			&C\e^{\frac{2}{3}- \sigma}\norm{u_0}_{H^2(\Omega)} \norm{\varphi}_{H^1(\Omega)} \qquad &\text{if } d =3, \\
			&C\e^{\frac{1}{2}-\sigma}\norm{u_0}_{H^2(\Omega)} \norm{\varphi}_{H^1(\Omega)} \qquad &\text{if } d = 4,\\
			&C\e^{\frac{1}{4} + \frac{3}{4(d-1)} - \sigma }\norm{u_0}_{H^2(\Omega)} \norm{\varphi}_{H^1(\Omega)} \qquad &\text{if } d \ge 5.
			\endaligned
			\right.
			\end{aligned}
		\end{equation*}
	Now, choosing $\varphi = w_\e$ in the above inequality and using (\ref{cond.ellipticity}) and (\ref{cond.elliptic.b}), we obtain
	\begin{equation*}
	\begin{aligned}
	\norm{\nabla w_\e}_{L^2(\Omega)}^2 + \norm{w_\e}_{L^2(\partial\Omega)}^2 \le \left\{
	\aligned
	&C\e^{\frac{1}{2}}\norm{u_0}_{H^2(\Omega)} \norm{w_\e}_{H^1(\Omega)} \qquad &\text{if } d =3, \\
	&C\e^{\frac{1}{2}-\sigma}\norm{u_0}_{H^2(\Omega)} \norm{w_\e}_{H^1(\Omega)} \qquad &\text{if } d = 4,\\
	&C\e^{\frac{1}{4} + \frac{3}{4(d-1)} - \sigma }\norm{u_0}_{H^2(\Omega)} \norm{w_\e}_{H^1(\Omega)} \qquad &\text{if } d \ge 5.
	\endaligned
	\right.
	\end{aligned}
	\end{equation*}
	The desired estimate follows from a simple observation $\norm{w_\e}_{H^1(\Omega)} \simeq \norm{\nabla w_\e}_{L^2(\Omega)} + \norm{w_\e}_{L^2(\partial\Omega)}$ and the H\"{o}lder inequality.

	(ii) The estimate (\ref{est.weH1.5d}) follows from Lemma \ref{lem.Iosc.H1} (i) and a similar argument.
\end{proof}

\begin{remark}\label{rmk.u0H2}
	Lemma \ref{lem.Iosc.H1} and Theorem \ref{thm.we.H1} shows that the estimate of $I_{\text{osc}}$ and the convergence rates of $w_\e$ are very sensitive to the regularity of $u_0$. For instance, $u_0\in H^2(\Omega;\RM)$ implies that $u_0\in W^{1,\frac{2d}{d-2}}(\Omega;\RM)$, which gives worse integrability as $d$ increases ($d\ge 5$) and hence causes lower convergence rates. On the other hand, $u_0 \in W^{1,d}(\Omega;\RM)$ is a very natural condition to guarantee good convergence rates for $d\ge 5$.
\end{remark}

\begin{lemma}\label{lem.5d}
		Assume $d\ge 5$. Let $\varphi \in H^2(\Omega;\RD)$ and $u_0\in W^{2,\frac{d}{2}}(\Omega;\RM)$.
	Then	
	\begin{equation*}
	\bigg| \int_{\Omega} A_\e \nabla w_\e \cdot \nabla \varphi + \int_{\partial\Omega} b_\e w_\e \cdot \varphi \bigg| \le C\e^{1-\sigma} \norm{u_0}_{W^{2,\frac{d}{2}}(\Omega)} \norm{\varphi}_{H^2(\Omega)},
	\end{equation*}
	for any $\sigma>0$.
\end{lemma}

\begin{proof}
	As before, we still have (\ref{eq.splitInt}) and (\ref{est.I123}), which yields
	\begin{equation}\label{est.I123.H2}
	|I_1| + |I_2| + |I_3| \le C\e \norm{u_0}_{H^2(\Omega)} \norm{\varphi}_{H^2(\Omega)}, \quad \text{for any } \varphi\in H^2(\Omega;\RM).
	\end{equation}
	Therefore, the main difficulty is to estimate
	\begin{equation*}
	I_{\text{osc}} = \int_{\partial\Omega} (\overline{b}^{\alpha\beta} - b^{\alpha\beta}(x/\e) - M_\e^{\alpha\beta}) u_0^\beta \cdot \varphi^\alpha.
	\end{equation*}
	By the equations of $v_\e^\beta$ in (\ref{eq.ve.b}) and the integration by parts, we have
	\begin{equation*}
	\begin{aligned}
	I_{\text{osc}}
	& = \int_{\partial\Omega} n\cdot \widehat{A}^{\alpha\gamma} \nabla v_\e^{\gamma\beta} u_0^\beta \cdot \varphi^\alpha \\
	& = \int_{\Omega} \nabla v_\e^{\gamma\beta}\cdot \widehat{A}^{*\gamma\alpha}\nabla( u_0^\beta \varphi^\alpha) \\
	& = - \int_{\Omega} v_\e^{\gamma\beta} \cdot \text{div}( \widehat{A}^{*\gamma\alpha}\nabla (u_0^\beta \varphi^\alpha) ) + \int_{\partial\Omega} v_\e^{\gamma\beta} (n\cdot \widehat{A}^{*\gamma\alpha}\nabla u_0^\beta) \varphi^\alpha + \int_{\partial\Omega} v_\e^{\gamma\beta}  u_0^\beta (n\cdot \widehat{A}^{*\gamma\alpha}\nabla \varphi^\alpha) \\
	& = J_1 + J_2 + J_3.
	\end{aligned}
	\end{equation*}
	
	To handle $J_1$, by (\ref{est.veLp.improved}) and the Sobolev inequality, we have
	\begin{equation}\label{est.J1.H2}
	\begin{aligned}
	|J_1| &\le C\norm{v_\e}_{L^{2+\sigma}(\Omega)} \norm{u_0}_{W^{2,\frac{d}{2}}(\Omega)} \norm{\varphi}_{H^2(\Omega)} \\
	& \le C\e^{1-\sigma}\norm{u_0}_{W^{2,\frac{d}{2}}(\Omega)} \norm{\varphi}_{H^2(\Omega)},
	\end{aligned}
	\end{equation}
	provided $d\ge 4$.
	On the other hand, if $d\ge 5$, (\ref{est.ve34d}) and the trace theorem implies
	\begin{equation}\label{est.J23.H2}
	\begin{aligned}
	|J_2| + |J_3| & \le C\norm{v_\e}_{L^{2+\sigma}(\partial\Omega)} \big( \norm{\nabla u_0 \varphi}_{W^{1,2-\sigma_1}(\Omega)} + \norm{ u_0 \nabla\varphi}_{W^{1,2-\sigma_1}(\Omega)} \big) \\
	& \le C\e^{1-\sigma} \norm{u_0}_{W^{2,\frac{d}{2}}(\Omega)} \norm{\varphi}_{H^2(\Omega)},
	\end{aligned}
	\end{equation}
	where $\sigma,\sigma_1$ are small and $\frac{1}{2+\sigma} + \frac{1}{2-\sigma_1} = 1$.
	This finishes the proof by combining (\ref{est.J1.H2}), (\ref{est.J23.H2}) and (\ref{est.I123.H2}).
\end{proof}

\begin{lemma}\label{lem.34d}
	Assume $d= 3$ or $4$. Let $u_0\in H^2\cap W^{1,\infty}(\Omega;\RM)$. Let $\varphi \in H^2(\Omega;\RD)$ be the weak solution of
			\begin{equation}\label{eq.varphi.G}
			\left\{
			\aligned
			-{\rm div}(\widehat{A}^* \nabla \varphi) & = G \quad &\text{ in }& \Omega,\\
			n\cdot \widehat{A}^*\nabla \varphi + \bar{b}^* \varphi & =0 \quad &\text{ on }& \partial\Omega.
			\endaligned
			\right.
			\end{equation}
	Then	
	\begin{equation*}
	\bigg| \int_{\Omega} A_\e \nabla w_\e \cdot \nabla \varphi + \int_{\partial\Omega} b_\e w_\e \cdot \varphi \bigg| \le  \left\{ 
	\begin{aligned}
	 &C\e^{\frac{7}{8}-\sigma} \big(\norm{u_0}_{W^{1,\infty}(\Omega)} + \norm{\nabla^2 u_0}_{L^2(\Omega)}\big) \norm{\varphi}_{H^2(\Omega)},\\
	 &C\e^{1-\sigma} \big(\norm{u_0}_{W^{1,\infty}(\Omega)} + \norm{\nabla^2 u_0}_{L^2(\Omega)}\big) \norm{\varphi}_{H^2(\Omega)},
	 \end{aligned}
	 \right.
	\end{equation*}
	for any $\sigma>0$.
\end{lemma}

\begin{proof}
	The proof follows a line of Lemma \ref{lem.5d}. Note that (\ref{est.J1.H2}) still holds if $d =4$. If $d = 3$, it is replaced by
	\begin{equation*}
	|J_1| \le C\e^{\frac{7}{8} - \sigma} \norm{u_0}_{H^2(\Omega)} \norm{\varphi}_{H^2(\Omega)},
	\end{equation*}
	thanks to (\ref{est.veL2.sharp}).
	
	The estimates for $J_2$ and $J_3$ will be sightly different since we only have a better rate for $\norm{v_\e}_{L^1(\partial\Omega)}$ if $d = 3$ or $4$, instead of $\norm{v_\e}_{L^2(\partial\Omega)}$. Precisely, if $d = 3$, the Sobolev embedding theorem implies that $\norm{\varphi}_{L^\infty(\Omega)} \le C\norm{\varphi}_{H^2(\Omega)}$. Thus, (\ref{est.ve34d}) implies
	\begin{equation*}
	|J_2| \le C\norm{v_\e}_{L^1(\partial\Omega)} \norm{\nabla u_0}_{L^\infty(\partial\Omega)} \norm{\varphi}_{L^\infty(\partial\Omega)} \le C\e^{1-\sigma} \norm{\nabla u_0}_{L^\infty(\partial\Omega)} \norm{\varphi}_{H^2(\Omega)}.
	\end{equation*}
	For $J_3$, we take advantage of the boundary condition of $\varphi$ in (\ref{eq.varphi.G}) and obtain
	\begin{equation*}
	J_3 = -\int_{\partial\Omega} v_\e^{\gamma\beta}  u_0^\beta (\bar{b}^* \varphi)^\gamma,
	\end{equation*}
	which yields
	\begin{equation*}
	|J_3| \le C\e^{1-\sigma} \norm{ u_0}_{L^\infty(\partial\Omega)} \norm{\varphi}_{H^2(\Omega)}.
	\end{equation*}
	
	If $d=4$, $\varphi$ is not bounded. Nevertheless, we still have
	\begin{equation*}
	\norm{\varphi}_{L^q(\partial\Omega)} \le C\norm{\varphi}_{H^2(\Omega)}, \qquad \text{for any } q<\infty.
	\end{equation*}
	Consequently,
	\begin{equation*}
	\begin{aligned}
		|J_2| + |J_3| &\le C\norm{v_\e}_{L^{q'}(\partial\Omega)} \norm{u_0}_{W^{1,\infty}(\partial\Omega)} \norm{\varphi}_{L^q(\partial\Omega)}\\
		& \le C\e^{1-\sigma} \norm{u_0}_{W^{1,\infty}(\partial\Omega)} \norm{\varphi}_{H^2(\Omega)},
	\end{aligned}
	\end{equation*}
	where in the last inequality, we have chosen $q$ sufficiently large (hence, $q'$ is sufficiently close to $1$). Note that $\norm{v_\e}_{L^{q'}} \le C\e^{1-\sigma}$ if $q'$ is close enough to $1$, due to (\ref{est.ve34d}). These complete the proof.
\end{proof}

Now, we are in a position to prove Theorem \ref{thm.L2Rate}.

\begin{proof}[Proof of Theorem \ref{thm.L2Rate}]
	The proof is based on a duality argument. For a given function $G\in L^2(\Omega;\RM)$, let $h_\e$ be the solution of
	\begin{equation*}
	\left\{
	\aligned
	-\text{div}( A^*(x/\e)\nabla h_\e) & = G \quad &\text{ in }& \Omega,\\
	n\cdot A^*(x/\e) \nabla h_\e + b^*(x/\e) h_\e & =0 \quad &\text{ on }& \partial\Omega.
	\endaligned
	\right.
	\end{equation*}
	Let $h_0$ be the corresponding homogenized equation, namely,
	\begin{equation}\label{eq.h0.Neumann}
	\left\{
	\aligned
	-\text{div}( \widehat{A}^* \nabla h_0) & = G \quad &\text{ in }& \Omega,\\
	n\cdot \widehat{A}^* \nabla h_0 + \bar{b}^* h_0 & =0 \quad &\text{ on }& \partial\Omega.
	\endaligned
	\right.
	\end{equation}
	Since $\Omega$ is smooth, the $H^2$ regularity theory for (\ref{eq.h0.Neumann}) implies $h_0\in H^2(\Omega;\RM)$ and $\norm{h_0}_{H^2(\Omega)} \le C\norm{G}_{L^2(\Omega)}$. It follows from Theorem \ref{thm.we.H1} (i) that
	\begin{equation}\label{est.he.H1}
	\norm{h_\e - h_0 - \e \chi(x/\e) S_\e( \eta_\e\nabla h_0) }_{H^1(\Omega)} \le \left\{
	\aligned
	&C\e^{\frac{1}{2}}\norm{h_0}_{H^2(\Omega)} \qquad &\text{if } d =3, \\
	&C\e^{\frac{1}{2}-\sigma}\norm{h_0}_{H^2(\Omega)}  \qquad &\text{if } d = 4,\\
	&C\e^{\frac{1}{4} + \frac{3}{4(d-1)} - \sigma }\norm{h_0}_{H^2(\Omega)} \qquad &\text{if } d \ge 5.
	\endaligned
	\right.
	\end{equation}
	
	Now, let $w_\e$ be defined as (\ref{eq.we}) and consider
	\begin{equation*}
	\begin{aligned}
	\int_{\Omega} w_\e \cdot G &= \int_{\Omega} A_\e \nabla w_\e \cdot \nabla h_\e - \int_{\partial\Omega} w_\e\cdot (n\cdot A^*(x/\e) \nabla h_\e) d\sigma \\
	& = \int_{\Omega} A_\e \nabla w_\e \cdot \nabla ( h_\e - h_0 - \e \chi(x/\e) S_\e( \eta_\e\nabla h_0) ) \\
	& \qquad + \int_{\partial\Omega}  w_\e \cdot b^*(x/\e) (h_\e - h_0) \cdot \\
	& \qquad + \int_{\Omega} A_\e \nabla w_\e \cdot \nabla h_0 + \int_{\partial\Omega} b(x/\e) w_\e \cdot h_0 d\sigma \\
	& \qquad + \int_{\Omega} A_\e \nabla w_\e \cdot \nabla (\e \chi(x/\e) S_\e( \eta_\e\nabla h_0) ) \\
	&= K_1 + K_2 + (K_3+K_4) + K_5.
	\end{aligned}
	\end{equation*}
	
	To estimate $K_1$, we use Theorem \ref{thm.we.H1}  and (\ref{est.he.H1}) to obtain
	\begin{equation*}
	\begin{aligned}
	|K_1| & \le \norm{w_\e}_{H^1(\Omega)} \norm{h_\e - h_0 - \e \chi(x/\e) S_\e( \eta_\e\nabla h_0) }_{H^1(\Omega)} \\
	& \le \left\{
	\aligned
	&C\e \norm{u_0}_{H^2(\Omega)} \norm{h_0}_{H^2(\Omega)} \qquad &\text{if } d =3, \\
	&C\e^{1-\sigma}\norm{u_0}_{H^2(\Omega)} \norm{h_0}_{H^2(\Omega)}  \qquad &\text{if } d = 4,\\
	&C\e^{\frac{3}{4} + \frac{3}{4(d-1)} - \sigma } \big( \norm{u_0}_{W^{1,d}(\Omega)} + \norm{u_0}_{H^2(\Omega)} \big) \norm{h_0}_{H^2(\Omega)} \qquad &\text{if } d \ge 5.
	\endaligned
	\right.
	\end{aligned}
	\end{equation*}
	To estimate $K_2$, note that $\e \chi(x/\e) S_\e( \eta_\e\nabla h_0)$ vanishes on the boundary. Then, $K_2$ has the same bound as $K_1$, by the trace theorem.

	Next, using Lemma \ref{lem.5d} and Lemma \ref{lem.34d}, we bound $K_3+K_4$ by
	\begin{equation*}
	\begin{aligned}
	|K_3 + K_4| &= \bigg| \int_{\Omega} A_\e \nabla w_\e \cdot \nabla h_0 + \int_{\partial\Omega} b(x/\e) w_\e \cdot h_0 d\sigma \bigg| \\
	&\le \left\{
	\aligned
	&C\e^{\frac{7}{8} - \sigma } \big( \norm{u_0}_{W^{1,\infty}(\Omega)} + \norm{u_0}_{H^2(\Omega)} \big) \norm{h_0}_{H^2(\Omega)} \qquad &\text{if } d =3, \\
	&C\e^{1-\sigma}\big( \norm{u_0}_{W^{1,\infty}(\Omega)} + \norm{u_0}_{H^2(\Omega)} \big) \norm{h_0}_{H^2(\Omega)}  \qquad &\text{if } d = 4,\\
	&C\e^{1 - \sigma } \norm{u_0}_{W^{2,\frac{d}{2}}(\Omega)} \norm{h_0}_{H^2(\Omega)} \qquad &\text{if } d \ge 5.
	\endaligned
	\right.
	\end{aligned}
	\end{equation*}
	
	Finally, to estimate $K_5$, we use the fact that $\e \chi(x/\e) S_\e( \eta_\e\nabla h_0)$ vanishes on a boundary layer with thickness $2\e$, if we choose the cut-off function $\eta_\e$ appropriately. In view of (\ref{eq.splitInt}) and (\ref{est.I123}), we have
	\begin{equation*}
	\begin{aligned}
	|K_5| & = \bigg| \int_{\Omega} A_\e \nabla w_\e \cdot \nabla (\e \chi(x/\e) S_\e( \eta_\e\nabla h_0) ) + \int_{\partial\Omega} b(x/\e) w_\e\cdot \e \chi(x/\e) S_\e( \eta_\e\nabla h_0) d\sigma| \\
	& \le C\e^{\frac{1}{2}} \norm{u_0}_{H^2(\Omega)} \big( \e^{\frac{1}{2}} \norm{\nabla (\e \chi(x/\e) S_\e( \eta_\e\nabla h_0))}_{L^2(\Omega)} + \norm{\nabla (\e \chi(x/\e) S_\e( \eta_\e\nabla h_0))}_{L^2(\Omega_{2\e})} \big) \\
	& = C\e \norm{u_0}_{H^2(\Omega)} \norm{\nabla (\e \chi(x/\e) S_\e( \eta_\e\nabla h_0))}_{L^2(\Omega)} \\
	& \le C\e \norm{u_0}_{H^2(\Omega)} \norm{h_0}_{H^2(\Omega)},
	\end{aligned}
	\end{equation*}
	where in the last inequality, we have used a property of $S_\e$; see \cite[Proposition 3.1.5]{ShenBook} or \cite[Lemma 2.1]{Shen17}.
	
	As a result, we arrive at
	\begin{equation}\label{est.we.G}
	\bigg| \int_{\Omega} w_\e \cdot G \bigg| \le \left\{
	\aligned
	&C\e^{\frac{7}{8} - \sigma } \big( \norm{u_0}_{W^{1,\infty}(\Omega)} + \norm{u_0}_{H^2(\Omega)} \big) \norm{h_0}_{H^2(\Omega)} \qquad &\text{if } d =3, \\
	&C\e^{1-\sigma}\big( \norm{u_0}_{W^{1,\infty}(\Omega)} + \norm{u_0}_{H^2(\Omega)} \big) \norm{h_0}_{H^2(\Omega)}  \qquad &\text{if } d = 4,\\
	&C\e^{\frac{3}{4} + \frac{3}{4(d-1)} - \sigma } \norm{u_0}_{W^{2,\frac{d}{2}}(\Omega)} \norm{h_0}_{H^2(\Omega)} \qquad &\text{if } d \ge 5.
	\endaligned
	\right.
	\end{equation}
	Recall that $\norm{h_0}_{H^2(\Omega)} \le C\norm{G}_{L^2(\Omega)}$. Then, by duality, we see that $\norm{w_\e}_{L^2(\Omega)}$ is bounded by the right-hand side of (\ref{est.we.G}), which yields the desired estimate (\ref{est.L2Rates}), since $\e \chi(x/\e) S_\e( \eta_\e\nabla u_0)$ is clearly bounded by $C\e\norm{u_0}_{H^1(\Omega)}$.
\end{proof}

\begin{remark}\label{rmk.Lprate}
	In this paper, we focus only on the estimate of $\norm{u_\e - u_0}_{L^2(\Omega)}$. However, the sharp convergence rates may be obtained if we consider $\norm{u_\e - u_0}_{L^p(\Omega)}$ with some $p\in(1,2)$. For instance, if $d = 3$, (\ref{est.veLp.improved}) implies that $\norm{v_\e}_{L^\frac{3}{2}(\Omega)} \le C\e^{1-\sigma}$. This leads to the improvement for the estimate of $J_1$ in Lemma \ref{lem.34d} and eventually yields 
	\begin{equation*}
	\norm{u_\e - u_0}_{L^\frac{3}{2}(\Omega)} \le C\e^{1-\sigma} ( \norm{u_0}_{W^{1,\infty}(\Omega)} + \norm{u_0}_{H^2(\Omega)} ).
	\end{equation*}
	Similar results are also true for $d\ge 5$.
\end{remark}

\medskip

\noindent{\bf Acknowledgment.}
Jun Geng was Supported in part by the NNSF of China (no.11571152) and Fundamental
Research Funds for the Central Universities (lzujbky-2017-161). Jinping Zhuge is supported in part by National Science Foundation grant DMS-1600520. Both of the authors would like to thank Professor
Zhongwei Shen for helpful discussions and suggestions.

\bibliographystyle{abbrv}
\bibliography{mybib}

\bigskip

\begin{flushleft}
Jun Geng\\
School of Mathematics and Statistics,
Lanzhou University,
Lanzhou, P.R. China\\
E-mail:gengjun@lzu.edu.cn
\end{flushleft}

\begin{flushleft}
Jinping Zhuge\\
Department of Mathematics,
University of Kentucky,
Lexington, Kentucky, 40506,
USA.\\
E-mail: jinping.zhuge@uky.edu
\end{flushleft}

\noindent

\end{document}